\documentclass{article}

\usepackage{a4wide}
\usepackage{amsmath}
\usepackage{amssymb}
\usepackage{amsthm}
\usepackage{amsfonts}
\usepackage{graphicx}
\usepackage{array}
\usepackage{calc}
\usepackage{url}
\numberwithin{equation}{section}

\newtheorem{lm}{Lemma}
\newtheorem{tw}{Theorem}

\newtheorem{wn}{Corollary}
\newtheorem{uw}{Remark}

\def\Xint#1{\mathchoice
{\XXint\displaystyle\textstyle{#1}}%
{\XXint\textstyle\scriptstyle{#1}}%
{\XXint\scriptstyle\scriptscriptstyle{#1}}%
{\XXint\scriptscriptstyle\scriptscriptstyle{#1}}%
\!\int}
\def\XXint#1#2#3{{\setbox0=\hbox{$#1{#2#3}{\int}$ }
\vcenter{\hbox{$#2#3$ }}\kern-.59\wd0}}
\def\cpint{\Xint-}

\newcommand{\ds}{\displaystyle}
\newcommand{\htbs}{\hspace*{-0.1ex}}
\newcommand{\htbbs}{\hspace*{-0.5em}}
\newcommand{\hts}{\hspace*{0.2ex}}
\newcommand{\hpo}{\hphantom{0}}

\newcommand{\hpm}{\hphantom{-}}

\newcolumntype{C}[1]{>{\centering\arraybackslash}m{#1}}

\newcommand{\leqx}{\hts\,{}^{_{_{_{\displaystyle{<}}}}}
                   \hspace{-2.3ex} {}_{_{\displaystyle{\sim}}}\,}

\newcommand{\al}{\gamma}
\newcommand{\eps}{\varepsilon}
\newcommand{\fl}[1]{#1^{[\eps]\htbs}}
\newcommand{\pal}{\htbs\htbs\htbs\htbs+\htbs\htbs\htbs\al}
\newcommand{\tpx}{\tau\htbs\htbs\htbs+\htbs\htbs x}
\newcommand{\tmx}{\tau\htbs\htbs\htbs-\htbs\htbs x}
\newcommand{\e}{\mathsf{e}}
\newcommand{\ff}{\psi}

\newcommand{\mfl}[2]{#1\cdot\htbs\htbs10^{#2}}
\newcommand{\tfl}[2]{$#1\cdot\htbs\htbs10^{#2}$}
\newcommand{\tflb}[2]{[\,$#1\cdot\htbs\htbs10^{#2}$\,]}
\newcommand{\tflu}[2]{\underline{$#1\cdot\htbs\htbs10^{#2}$}}
\newcommand{\na}{\makebox[\widthof{\tfl{0.0}{-10}}]{\sf{N/A}}}
\newcommand{\sna}{\makebox[\widthof{$1.0T$}]{\sf{N/A}}}
\newcommand{\dar}{${}^{\downarrow}$}
\newcommand{\uar}{${}_{\uparrow}$}

\newcommand{\Maple}{\emph{Maple}}
\newcommand{\Matlab}{\emph{Matlab}}

\newcommand{\quadcc}{{\tt quadcc}}
\newcommand{\quadgk}{{\tt quadgk}}

\newcommand{\qawc}{{\tt qawc}}

\begin{document}

\title{Roundoff errors in the problem of
       computing Cauchy principal value integrals}

\author{Pawe\l ~Keller, Iwona Wr\'obel}

%\institute{P.\ Keller \at
%           Institute of Computer Science, University of Wroc\l{}aw,
%           ul.~Joliot-Curie 15, 50-383 Wroc\l{}aw, Poland \\
%           Fax: +48-71-3757801 \\
%           \email{Pawel.Keller@ii.uni.wroc.pl}}

\date{\vspace*{-1ex}
\scriptsize Faculty of Mathematics and Information Science, Warsaw University of Technology}

\maketitle

\begin{abstract}
We investigate the possibility of fast, accurate and
reliable computation of the Cauchy principal value integrals
$\cpint_{a\vphantom{_0}}^{b}\htbs f(x)(x-\tau)^{-1} dx$ $(a < \tau < b)$
using standard adaptive quadratures. In order to properly control the error
tolerance for the adaptive quadrature and to obtain a~reliable estimation
of the approximation error, we research the possible influence of round-off
errors on the computed result. As the numerical experiments confirm, the proposed
method can successfully compete with other algorithms for computing such type
integrals. Moreover, the presented method is very easy to implement on any
system equipped with a reliable adaptive integration subroutine.\\[1ex]
\end{abstract}

%\linenumbers 
%%%%%%%%%%%%%%%%%%%%%%%%%%%%%%%%%%%%%%%%%%%%%%%%%%%%%%%%%%%%%%%%%%%%%%%%%%%%%%%%
%%%%%%%%%%%%%%%%%%%%%%%%%%%%%%%%%%%%%%%%%%%%%%%%%%%%%%%%%%%%%%%%%%%%%%%%%%%%%%%%
\section{Introduction} \label{SecIntro}
\setcounter{equation}{0}
%%%%%%%%%%%%%%%%%%%%%%%%%%%%%%%%%%%%%%%%%%%%%%%%%%%%%%%%%%%%%%%%%%%%%%%%%%%%%%%%
%%%%%%%%%%%%%%%%%%%%%%%%%%%%%%%%%%%%%%%%%%%%%%%%%%%%%%%%%%%%%%%%%%%%%%%%%%%%%%%%

We consider the problem of numerical evaluation of the Cauchy
principal value integral
\begin{equation}
  I_{a,b,\tau}(f) \,:=\,
    \cpint_{a}^{b}\frac{f(x)}{x-\tau}\hts dx \,=
      \lim\limits_{\mu\to 0^{+}}
        \bigg(\int_{a}^{\tau-\mu} + \int_{\tau+\mu}^{b} \bigg)
          \frac{f(x)}{x-\tau}\hts dx ,
\label{CPInt}
\end{equation}
where $\tau\in(a,b)$, and the function $f$ has bounded first derivative.
In general, the integral (\ref{CPInt}) exists if $f$ is H\"older continuous
(cf.\ \cite[\S1.6]{Davis}). The integrals of this type appear in many practical
problems related to aerodynamics, wave propagation or fluid and fracture mechanics,
mostly with relation to solving singular integral equations.

A great many papers devoted to the problem of numerical evaluation of the integrals
of the form (\ref{CPInt}) have been published so far. Some of them are
\cite{Elliot,HTCauchy,Longman,Paget,Piess,Stewart}. A nice survey on the
subject, along with a large number of references, is presented
in \cite[\S\,2.12.8]{Davis}.

Even though so many algorithms have been known for a quite long time, the subroutines
for computing the integrals of the type (\ref{CPInt}) are not commonly available
in systems for scientific computations. This is probably because most of the
methods assume some properties of the function $f$, e.g.\ that $f$ is analytic,
$f$ is smooth enough, etc. On the other hand, almost every system for scientific
computations is equipped with one or more subroutines for automatic
computation of the integrals
\begin{equation}
  \int_a^b \ff(x)\hts dx\hts .
\label{Int}
\end{equation}
These algorithms, usually called \emph{adaptive quadratures},
can compute the integrals of the form (\ref{Int}) for a very wide
range of integrands.

The natural question arises: can an adaptive quadrature be successfully
used to numerically compute the integral (\ref{CPInt})? In this paper,
we search for the positive answer to that question. Our goal is to provide
tools that will allow to use an existing adaptive quadrature to compute
accurate approximation to the integral (\ref{CPInt}), and also will
help to obtain a reliable error estimation of the computed result.
A usual automatic (adaptive) quadrature performs the computations
until the error tolerance provided by a user (or the default, fixed
one) is met. The numerical value of the Cauchy principal values integral is
sometimes quite sensitive to the influence of round-off errors. If a user
wishes high accuracy of the approximation, it may be difficult to select a proper
error tolerance (if set too small, then the computation time may considerably
increase or even the quadrature may fail to deliver reliable approximation).
Thus, in the method proposed in this paper, the error tolerance is trimmed
at the safe level. In particular, if the tolerance is set to zero, the
algorithm is expected to compute the best (or rather near the best,
within reasonable limits) possible, approximation to the
integral (\ref{CPInt}) for given parameters $f$ and $\tau$.

The next section contains the formulation of, what
may be called, the basic form of our algorithm. We apply
two known analytical transformations that convert the integral (\ref{CPInt})
into the sum of two nonsingular integrals. In Section \ref{SecAda}, we briefly
describe the idea of adaptive quadratures. In the subsequent section,
we discuss the problem of possible loss of significant digits and
the influence of round-off errors on the approximation to a Cauchy
principal value integral computed using the formula derived
in Section \ref{SecAlg}. The obtained error estimates
are crucial for the efficiency and reliability of the proposed method.
As we briefly demonstrate, these estimates can be used to increase reliability
of other algorithms for computing integrals of the form (\ref{CPInt}).
In Section \ref{SecNum}, we present many numerical examples to validate
the usefulness of the proposed algorithm. The method is thoroughly tested
with three different adaptive quadratures: the built-in adaptive quadrature
of the {\Maple} system, the one included in {\Matlab}, and the one
presented in \cite{GonnetAlg}.

%%%%%%%%%%%%%%%%%%%%%%%%%%%%%%%%%%%%%%%%%%%%%%%%%%%%%%%%%%%%%%%%%%%%%%%%%%%%%%%%
%%%%%%%%%%%%%%%%%%%%%%%%%%%%%%%%%%%%%%%%%%%%%%%%%%%%%%%%%%%%%%%%%%%%%%%%%%%%%%%%

%%%%%%%%%%%%%%%%%%%%%%%%%%%%%%%%%%%%%%%%%%%%%%%%%%%%%%%%%%%%%%%%%%%%%%%%%%%%%%%%
%%%%%%%%%%%%%%%%%%%%%%%%%%%%%%%%%%%%%%%%%%%%%%%%%%%%%%%%%%%%%%%%%%%%%%%%%%%%%%%%
\section{Analytical transformations} \label{SecAlg}
\setcounter{equation}{0}
%%%%%%%%%%%%%%%%%%%%%%%%%%%%%%%%%%%%%%%%%%%%%%%%%%%%%%%%%%%%%%%%%%%%%%%%%%%%%%%%
%%%%%%%%%%%%%%%%%%%%%%%%%%%%%%%%%%%%%%%%%%%%%%%%%%%%%%%%%%%%%%%%%%%%%%%%%%%%%%%%

Without loss of generality, we may restrict our attention to the case $a=-1$
and $b=1$. The computation of the Cauchy principal value integral
\begin{equation}
  I_{\tau}(f) \,\equiv\, I_{\tau,-1,1}(f) \,=\,
    \cpint_{-1}^{1}\frac{f(x)}{x-\tau}\hts dx 
\label{CPInt1}
\end{equation}
may at first seem quite easy, if we observe that by a simple change
of variables, setting
\begin{equation}
\delta = \min\{1+\tau,1-\tau\} ,
\label{DeltaDef}
\end{equation}
we obtain
\begin{equation}
\begin{array}{l}\ds
  I{_\tau}(f) \,=\,
  \int_{|x-\tau|\geq\delta} \frac{f(x)}{x-\tau}\hts dx +\,
  \cpint_{\tau-\delta}^{\tau+\delta} \frac{f(x)}{x-\tau}\hts dx \\[3.5ex]\ds
  \hphantom{I{_\tau}(f)} \,=\,
  \int_{|x-\tau|\geq\delta}  \frac{f(x)}{x-\tau}\hts dx +
  \int_{0}^{\delta} \frac{f(\tau+x)-f(\tau-x)}{x}\hts dx ,
\end{array}
\label{Formula1}
\end{equation}
where we use the convention that
$\int_{|x-\tau|\geq\delta} \equiv \int_{\tau+\delta}^{1}$, if $\delta = 1+\tau$, and
$\int_{|x-\tau|\geq\delta} \equiv \int_{-1}^{\tau-\delta}$, if $\delta = 1-\tau$.

The formula (\ref{Formula1}) was applied for the first time by Longman in \cite{Longman},
and it was derived by splitting the function $f$ into the odd and even parts. Both
integrals on the right hand side of (\ref{Formula1}) exist in the Riemann sense.
We should note that if the function $f$ has bounded first derivative
in the neighbourhood of $\tau$, then the second integral is not even singular
(unless $f$ has singularities itself).

The first integral on the right hand side of (\ref{Formula1}) was commonly not paid
attention to, as it is always a~proper one. However, if $|\tau|$ is close to $1$,
then this integral is a near-singular one, and standard quadratures may fail
when applied directly.

In many algorithms, another transformation of the integral $I_{\tau}(f)$ is used,
usually being called \emph{subtracting out the singularity}. We have
\begin{equation}
\begin{array}{l}\ds
  I{_\tau}(f) \,=\,
  \cpint_{-1}^{1}\frac{f(\tau)}{x-\tau}\hts dx +
    \int_{-1}^{1}\frac{f(x)-f(\tau)}{x-\tau}\hts dx \\[3.2ex]\ds
  \hphantom{I{_\tau}(f)} \,=\,
  f(\tau)\,\log\frac{1-\tau}{1+\tau} + \int_{-1}^{1}\frac{f(x)-f(\tau)}{x-\tau}\hts dx\, .
\end{array}
\label{Formula2}
\end{equation}
A direct application of the above formula is commonly not recommended
(see, e.g., \cite{HTCauchy,Monte}) due to possible severe cancellation,
if a quadrature node happens to be very close to $\tau$.

The authors, however, have never seen the two above approaches
being put together\footnote{In fact, the formula (\ref{Formula2}) was already
presented in the numerical section of \cite{KellerCPOsc}, however, it was done
by referring to the preliminary results \cite{KellerArxiv} that laid
the basis of the present paper.}.
From (\ref{Formula1}) and (\ref{Formula2}) we immediately obtain
\begin{equation}
  I{_\tau}(f) \,=\,
  f(\tau)\,\log\frac{1-\tau}{1+\tau} +
  \int_{|x-\tau|\geq\delta} g(x)\hts dx +
  \int_{0}^{\delta} h(x)\hts dx ,
\vspace{-0.5ex}
\label{MainFormula}
\end{equation}
where
\begin{equation}
  g(x) = \frac{f(x)-f(\tau)}{x-\tau} \qquad\,\, \mathrm{and} \qquad\,\,
  h(x) = \frac{f(\tau+x)-f(\tau-x)}{x} .
\vspace{0.3ex}
\label{ghDef}
\end{equation}
If the function $f$ has bounded first derivative, none of the integrals on the
right hand side of (\ref{MainFormula}) is singular or near-singular (unless $f$ itself
has singularities just outside the interval $[-1,1]$), and, when the first integral
of (\ref{MainFormula}) is approximated numerically, the distance between $\tau$
and any of the quadrature nodes is never smaller than $\delta$.

It is clear that some loss of significance may occur when computing the
values of the function $g$ and $h$ defined above. The influence of the round-off
errors on the accuracy of the approximation (\ref{MainFormula}) to the integral
(\ref{CPInt1}) is thoroughly discussed in Section \ref{SecErr}.

%%%%%%%%%%%%%%%%%%%%%%%%%%%%%%%%%%%%%%%%%%%%%%%%%%%%%%%%%%%%%%%%%%%%%%%%%%%%%%%%
%%%%%%%%%%%%%%%%%%%%%%%%%%%%%%%%%%%%%%%%%%%%%%%%%%%%%%%%%%%%%%%%%%%%%%%%%%%%%%%%

%%%%%%%%%%%%%%%%%%%%%%%%%%%%%%%%%%%%%%%%%%%%%%%%%%%%%%%%%%%%%%%%%%%%%%%%%%%%%%%%
%%%%%%%%%%%%%%%%%%%%%%%%%%%%%%%%%%%%%%%%%%%%%%%%%%%%%%%%%%%%%%%%%%%%%%%%%%%%%%%%
\section{A very short story on adaptive quadratures} \label{SecAda}
\setcounter{equation}{0}
%%%%%%%%%%%%%%%%%%%%%%%%%%%%%%%%%%%%%%%%%%%%%%%%%%%%%%%%%%%%%%%%%%%%%%%%%%%%%%%%
%%%%%%%%%%%%%%%%%%%%%%%%%%%%%%%%%%%%%%%%%%%%%%%%%%%%%%%%%%%%%%%%%%%%%%%%%%%%%%%%

The idea and basic rules of adaptive numerical integration (including some examples)
is nicely presented in \cite[\S6]{Davis} and \cite[\S5.2\,\&\,\S5.3]{ShampineBook}. A very
thorough study on the subject, together with the comprehensive history of
development of adaptive quadratures, is given in \cite{GonnetPhD}.

In general, adaptive quadratures are meant to be able to numerically approximate the
integrals of the form (\ref{Int}) for the widest possible range of integrands.
The only arguments required by an adaptive algorithm are the integrand and
the interval endpoints. The algorithm tries to approximate the integral
with an error less than some default error tolerance or the error
tolerance provided by the user. Many adaptive quadratures also report the
estimation of the approximation error.

The most common idea of the adaptive integration algorithm is to split the interval
of integration into subintervals of different lengths. The number of the subintervals
and their lengths depend on the behaviour of the integrand, and are determined
automatically during the computation process. Usually, the approximated value of
the integral over a single subinterval is computed using some fixed quadrature rule.
Along with the approximation, an error estimation has to be computed,
e.g., by the use of another quadrature rule and comparing the results.
The computations are terminated, when the estimated global error
of the adaptive quadrature is less than a prescribed tolerance.

Suppose that an adaptive integration algorithm is based on
the fixed linear quadrature rule
\begin{equation}
  Q(\ff) \equiv Q(\ff,0,1) = \sum\limits_{j=0}^{m} B_j\hts\ff(x_j)
    \quad (0\leq x_0 < x_1 < \dots < x_m \leq 1),
\label{QDef}
\end{equation}
which approximates the integral
\begin{equation*}
  \int_0^1 \ff(x) dx .
\end{equation*}
Then, the final approximation to the integral (\ref{Int}) computed by
the adaptive quadrature is given by
\begin{equation}
  A(\ff,a,b) = \sum\limits_{k=0}^{n-1} Q(\ff,s_k,s_{k+1}) =
    \sum\limits_{k=0}^{n-1} \sum\limits_{j=0}^{m} B_{kj}\hts\ff(x_{kj}) ,
\label{ASumSum}
\vspace{-1ex}
\end{equation}
where
\begin{equation*}
  B_{kj} = (s_{k+1}-s_k)B_j,
   \qquad x_{kj} = s_k + (s_{k+1}-s_k)x_j ,
\label{BxDef}
\end{equation*}
and
\begin{equation*}
  a = s_0 < s_1 < \dots < s_n = b
\label{cPointsDef}
\end{equation*}
are the automatically determined endpoints of the subintervals $[s_{k},s_{k+1}]$
($0\leq k<n$,\, $n\geq 1$). The values $s_k$ ($0\leq k\leq n$) and the integer $n$
depend on $\ff$, $a$, $b$ and the assumed error tolerance.

If the quadrature $Q$ defined in (\ref{QDef}) has the property $x_0=0$, $x_m=1$, then
it is called the \emph{closed type} one. If $x_0>0$ and $x_m<1$, then $Q$ is called
the quadrature rule of the \emph{open type}. In the similar way we shall
call the related adaptive quadrature.

When computing the integrals of the form (\ref{Int}), adaptive quadratures
are much more flexible compared to any fixed quadrature rule. However, it has
to be noted that sometimes an adaptive algorithm may fail to compute a reliable
approximation. In other words, the algorithm based on (\ref{MainFormula}) for
computing the integral (\ref{CPInt1}) will be, in general, as reliable as
the adaptive quadrature used. The problem of reliability of adaptive
quadratures was thoroughly discussed in \cite{GonnetAlg}
and \cite{GonnetErr}.

Some more information on a specific adaptive quadrature is provided
in Section \ref{SecNum}, when the given adaptive algorithm
is being considered.

Computing the integrals of the form (\ref{CPInt1}) by a simple application of
the formula (\ref{MainFormula}) may not lead to satisfactory results. The numerical values of the
Cauchy principal value integrals are, in some cases, very sensitive to round-off errors. Therefore,
the possible to achieve accuracy may be considerably limited, when computations are performed in a fixed
precision floating point arithmetic. If we set too small error tolerance for the adaptive quadrature,
the algorithm may work very long, usually being unable to reach the requested accuracy. Because of the
influence of the round-off errors, the error estimation reported by the adaptive algorithm may
also be false. Thus, it is very important to have some good tools that would modify the error
tolerance, if necessary, and correct the error estimation computed by
the adaptive quadrature.

%%%%%%%%%%%%%%%%%%%%%%%%%%%%%%%%%%%%%%%%%%%%%%%%%%%%%%%%%%%%%%%%%%%%%%%%%%%%%%%%
%%%%%%%%%%%%%%%%%%%%%%%%%%%%%%%%%%%%%%%%%%%%%%%%%%%%%%%%%%%%%%%%%%%%%%%%%%%%%%%%

%%%%%%%%%%%%%%%%%%%%%%%%%%%%%%%%%%%%%%%%%%%%%%%%%%%%%%%%%%%%%%%%%%%%%%%%%%%%%%%%
%%%%%%%%%%%%%%%%%%%%%%%%%%%%%%%%%%%%%%%%%%%%%%%%%%%%%%%%%%%%%%%%%%%%%%%%%%%%%%%%
\section{Estimating round-off errors} \label{SecErr}
\setcounter{equation}{0}
%%%%%%%%%%%%%%%%%%%%%%%%%%%%%%%%%%%%%%%%%%%%%%%%%%%%%%%%%%%%%%%%%%%%%%%%%%%%%%%%
%%%%%%%%%%%%%%%%%%%%%%%%%%%%%%%%%%%%%%%%%%%%%%%%%%%%%%%%%%%%%%%%%%%%%%%%%%%%%%%%

It is quite obvious that even small changes of the values of the
function~$f$ near the singularity point $\tau$ may considerably change the
value of the integral (\ref{CPInt1}). Also, when using the formula (\ref{MainFormula}),
an additional error, due to the computation of values of the functions $g$ and $h$,
may appear. In this section, we derive estimates of the round-off errors
that will allow to control the performance of an adaptive quadrature,
and will provide additional information on the accuracy of the
computed approximation to the integral~(\ref{CPInt1}).

Let us denote by $\eps$ the precision of the arithmetic used,
and by $\fl{g}(x)$ and $\fl{h}(x)$ the numerically computed values
of the functions $g$ and $h$, defined in~(\ref{ghDef}), at the point~$x$
(the argument $x$, and also the parameter $\tau$, may not have the exact representation
in the floating point arithmetic). For simplicity, we will temporarily assume that
the values of the function~$f$ in (\ref{ghDef}) are computed exactly. We shall
write $a\approx b$, if $a = b\big(1+O(\eps)\big)$, and $a\leqx b$,
if $a<b$ or $a\approx b$.

\begin{lm} \label{l1}
If we define
\begin{equation}
  D_{1,f} = \max_{|x|\leq 1} |f'(x)| ,
\label{D1Def}
\end{equation}
then for all $0<|x|\leq \delta$, where $\delta$ is defined in (\ref{DeltaDef}),
\begin{equation}
  | \fl{h}(x) - h(x) |\hts \leqx\, \frac{4\hts\eps D_{1,f}}{x} + 4\hts\eps D_{1,f} .
\label{hErrEst}
\end{equation}
\end{lm}
\begin{proof} We have
\begin{equation}
  \fl{h}(x) = \frac{ f\big((\tau(1\pal_1) + x(1\pal_2))(1\pal_3)\big) -
                     f\big((\tau(1\pal_1) - x(1\pal_2))(1\pal_4)\big) }
                   { x(1\pal_5) } ,
\label{hEps}
\end{equation}
where $|\al_1|,|\al_2|,|\al_3|,|\al_4|\leq\eps$ and $|\al_5|\leqx 2\eps$. Further,
\begin{equation*}
  (\tau(1\pal_1) + x(1\pal_2))(1\pal_3) \,=\, \tau + x + \beta
\end{equation*}
for some $\beta$ satisfying
\begin{equation*}
  |\beta| \,\approx\, \big|(\al_1+\al_3)\tau + (\al_2+\al_3)x\big| \,\leq\, 2\eps ,
\end{equation*}
as, by definition of $\delta$ given in (\ref{DeltaDef}), we have $|\tau|+|x|\leq 1$,
if $|x|\leq \delta$. This immediately implies that
\begin{equation}
  |f\big((\tau(1\pal_1) + x(1\pal_2))(1\pal_3)\big) - f(\tau+x)|
    \,\leq\, |\beta| D_{1,f} \,\leqx\, 2\hts\eps D_{1,f} .
\label{pest}
\end{equation}
Analogously,
\begin{equation}
  |f\big((\tau(1\pal_1) - x(1\pal_2))(1\pal_4)\big) - f(\tau-x)| \,\leqx\, 2\eps D_{1,f} .
\label{mest}
\end{equation}
Now, from (\ref{hEps}), (\ref{pest}) and (\ref{mest}) we obtain
\begin{align*}
  | \fl{h}(x) - h(x) |
    &\,=\, \bigg| \fl{h}(x) - \frac{\big(f(\tpx)-f(\tmx)\big)\hts(1\pal_5)}{x(1\pal_5)} \bigg|\\[0.5ex]
    &\,=\, \bigg| \fl{h}(x) - \frac{f(\tpx)-f(\tmx)}{x(1\pal_5)}
                         - \frac{\al_5}{1\pal_5}\cdot\frac{f(\tpx)-f(\tmx)}{x} \bigg|\\[0.5ex]
    &\,\leqx\, \frac{4\hts\eps D_{1,f}}{x} + 4\hts\eps D_{1,f} .
\end{align*}
\end{proof}

\begin{lm} \label{l2}
If $|x-\tau|\geq 16\eps$, then
\begin{equation}
  | \fl{g}(x) - g(x) | \,\leqx\, \frac{8\hts\eps D_{1,f}}{|x-\tau|} .
\label{gErrEst}
\end{equation}
\end{lm}
\begin{proof}
Similarly as in (\ref{hEps}), we have
\begin{equation}
 \fl{g}(x) = \frac{f(x(1\pal_1))-f(\tau(1\pal_2))}{(x(1\pal_1)-\tau(1\pal_2))(1\pal_3)} ,
\label{gEps}
\end{equation}
where $|\al_1|,|\al_2|\leq\eps$ and $|\al_3|\leqx 2\eps$. Moreover,
\begin{equation}
  (x(1\pal_1)-\tau(1\pal_2))(1\pal_3)-(x-\tau) \,=\, \beta
\label{BetaDef}
\end{equation}
for some $\beta\approx x\al_1 + \tau\al_2 + (x-\tau)\al_3$ that satisfies
\begin{equation}
  |\beta| \,\leqx\, \left\{
    \begin{array}{ll}
      4\eps\hts, & \mathrm{if}\quad |x-\tau| \leq 1, \\
      6\eps\hts, & \mathrm{otherwise} .
    \end{array}
  \right.
\label{BetaEstim}
\end{equation}
From (\ref{BetaDef}), we immediately obtain, that
\begin{equation}
  (x(1\pal_1)-\tau(1\pal_2))(1\pal_3)\, =\, (x-\tau)\Big(1+\frac{\beta}{x-\tau}\Big) .
\label{DenomAlfForm}
\end{equation}
In addition, as $|x|,|\tau|\leq 1$, we have
\begin{equation*}
  |f(x(1\pal_1)) - f(x)| \leq
    \eps D_{1,f} \quad\mathrm{and}\quad |f(\tau(1\pal_2)) - f(\tau)| \leq \eps D_{1,f} ,
\end{equation*}
which, together with (\ref{gEps}), (\ref{DenomAlfForm}), (\ref{BetaEstim}),
and $|x-\tau|\geq 16\eps$, gives
\begin{align*}
  |\fl{g}(x) - g(x) | &\,=\,
       \Bigg| \fl{g}(x) - 
           \frac{f(x)-f(\tau)}{(x-\tau)\Big(1+\frac{\beta}{x-\tau}\Big)}
               \Big(1+\frac{\beta}{x-\tau}\Big) \Bigg|\\[0.0ex]
  &\,=\, \Bigg| \fl{g}(x) - 
        \frac{f(x)-f(\tau)}{(x-\tau)\Big(1+\frac{\beta}{x-\tau}\Big)}
        - \frac{\beta\,{\ds\frac{f(x)-f(\tau)}{x-\tau}}}
               {(x-\tau)\Big(1+\frac{\beta}{x-\tau}\Big)} \Bigg|\\[2.0ex]
  &\,\leq\, \frac{2\hts\eps D_{1,f} + \beta D_{1,f}}{|x-\tau|\Big(1+\frac{\beta}{x-\tau}\Big)}
   \,\leqx\, \frac{8\hts\eps D_{1,f}}{|x-\tau|} .
\end{align*}
\end{proof}

\begin{uw} The round-off error estimates of Lemmas \ref{l1} and \ref{l2}
were obtained under the assumption that $D_{1,f} < \infty$, and they can allow us to
derive a global round-off error estimation for the integral (\ref{CPInt1}) computed
using (\ref{MainFormula}), which does not depend on $\tau$. It is easy to verify that
the inequalities (\ref{hErrEst}) and (\ref{gErrEst}) could have been be replaced
by more precise and much more complex estimates,
\begin{equation}
  | \fl{h}(x) - h(x) | \hts\leqx\, \frac{2\hts\eps}{x}\big(f'(\tau+x)+f'(\tau-x)\big)
                       \,+\, 2\hts\eps\big(\hat{f}(\tau+x,\tau)+\hat{f}(\tau-x,\tau)\big),
\label{hErrEstPrec}
\end{equation}
and
\begin{equation}\hspace*{-1.75ex}
  |\fl{g}(x) - g(x) | \,\leqx \left\{
    \begin{array}{ll}
      \ds\frac{4\hts\eps}{3|x-\tau|}
            \big(f'(x)+f'(\tau)+4\hat{f}(x,\tau)\big), & \mathit{if}\quad |x-\tau| \leq 1, \\[2.5ex]
      \ds\hphantom{3}\frac{\eps}{|x-\tau|}\big(f'(x)+f'(\tau)+6\hat{f}(x,\tau)\big), & \mathit{otherwise} ,
    \end{array}\right.
\vspace{1ex}
\label{gErrEstPrec}
\end{equation}
respectively, where $\hat{f}(x,y)=(x-y)^{-1}\big(f(x)-f(y)\big)$, if, e.g., $f''$ is bounded
in the $2\eps$-neighbourhood of $\hts\tau$, $x$, and $\tau\pm x$. The above inequalities, however,
appear to lead to no reasonably simple and usable global round-off error estimation. Therefore,
in the incoming part of the paper, we shall derive uniform ($\tau$-independent) global
theoretical estimates of the influence of the round-off errors, by applying (\ref{hErrEst})
and (\ref{gErrEst}). Still, in the implementation of the proposed method, we shall
make some use of the inequalities (\ref{hErrEstPrec})--(\ref{gErrEstPrec}).
\end{uw}

In both lemmas we assumed that the values of the function $f$ are computed exactly.
In practice, this is obviously not true. However, if the computation of $f(x)$ for every
$x\in[-1,1]$ is numerically backward stable, i.e.
\begin{equation}
  \fl{f}(x) = f(x(1+\gamma)),
\label{BkStab}
\end{equation}
where $\fl{f}(x)$ denotes the numerically computed value of $f(x)$,
and $|\gamma|\leq \eps K_1$ for some small positive $K_1$, then all the above results
remain true, only the constant factors change. Moreover, regarding the fact that in the
proofs of lemmas \ref{l1} and \ref{l2} we have bounded the errors $|\eps x|$ and $|\eps\tau|$
by $|\eps|$ (as $|x|,|\tau|\leq 1$), we may replace (\ref{BkStab}) with the less
restrictive condition,
\begin{equation*}
  \fl{f}(x) = f(x + \gamma) \qquad (-1\leq x\leq 1),
\label{BkAbsStab}
\end{equation*}
which better reflects the real situations, e.g., if the function $f$ depends on
the shifted argument. One may suggest that it is more practical to assume that
\begin{equation*}
  \fl{f}(x) = f(x + \gamma) (1+\xi) \qquad (-1\leq x\leq 1),
\label{BkAbsStabExt}
\end{equation*}
where $|\xi|\leq \eps K_0$ for some small $K_0>0$. In a similar manner as we have
proved Lemmas \ref{l1} and~\ref{l2}, it can be shown that in the above case the inequalities
(\ref{hErrEst}) and (\ref{gErrEst}) remain true, if the factor $D_{1,f}$ is replaced with
\begin{equation}
  D_{\htbs f} \,:=\, D_{1,f} + \frac{1}{2}(K_1 D_{1,f} + K_0 D_{0,f}) ,
\vspace{-0.5ex}
\label{DfDef}
\end{equation}
where
\begin{equation}
  D_{0,f} = \max_{|x|\leq 1} |f(x)| .
\label{D0Def}
\end{equation}

%%%%%%%%%%%%%%%%%%%%%%%%%%%%%%%%%%%%%%%%%%%%%%%%%%%%%%%%%%%%%%%%%%%%%%%%%%%%%%%%
\subsection{The first approach -- cutting of the singularity} \label{SecErrCut}

From (\ref{hErrEst}) and (\ref{gErrEst}) we immediately obtain
(formally) that\vspace{-1.5ex}
\begin{equation}
\renewcommand{\arraystretch}{2.5}
\begin{array}{rcl}
  \ds E_{\eps,0}(f) \htbbs&:=&\htbbs \ds
  \bigg|\int_{|x-\tau|\geq\delta} \big(\fl{g}(x)-g(x)\big)\hts dx
  + \int_{0}^{\delta} \big(\fl{h}(x)-h(x)\big)\hts dx\bigg| \\[0.5ex]
              \htbbs&\leq&\htbbs \ds
  \int_{|x-\tau|\geq\delta} \big|\fl{g}(x)-g(x)\big|\hts dx
  + \int_{0}^{\delta} \big|\fl{h}(x)-h(x)\big|\hts dx \\[0.5ex]
              \htbbs&\leqx&\htbbs
  \ds\int_{|x-\tau|\geq\delta}\frac{8\hts\eps D_{\htbs f}}{|x-\tau|}\hts dx \,+
     \int_{0}^{\delta}\frac{8\hts\eps D_{\htbs f}}{x}\hts dx \,\leq\,
     8\hts\eps D_{\htbs f} \int_{-1}^{1} \frac{1}{|x-\tau|}\hts dx \,\leq\,
     16\hts\eps D_{\htbs f} \int_{0}^{1} \frac{1}{x}\hts dx ,
\vspace{0.5ex}
\label{CutEst}
\end{array}
\end{equation}
which may not look like a promising formula for approximating the cumulative round-off error.
Thus, it seems quite natural to replace the integral $\int_{0}^{\delta}\htbs h(x)\hts dx$
in (\ref{MainFormula}) by $\int_{\mu}^{\delta}\htbs h(x)\hts dx$
for some very small value of $\mu$ ($0<\mu\leq\delta$).
In this case, we have
\begin{equation}
  E_{\eps,\mu}(f) \,\leqx\, 16\hts\eps D_{\htbs f} \log(\mu^{-1}) \hts .
\label{CutMuEst}
\end{equation}
By $A(\ff,a,b)$ we shall denote an approximated value of the integral (\ref{Int})
computed by some (adaptive) algorithm $A$. We will also define the
theoretical approximation error,
\begin{equation*}
  E_A(\ff,a,b) = \bigg| \int_a^b \ff(x)\hts dx - A(\ff,a,b) \bigg| .
\end{equation*}
Now, from (\ref{MainFormula}), (\ref{CutMuEst}), and
$\big|\htbs\htbs\int_{0}^{\mu} h(x)\hts dx\hts\big|\leq 2\mu D_{1,f}$, we obtain
\begin{equation}
  I_{\tau}(f) \hts=\hts f(\tau)\log\frac{1-\tau}{1+\tau} +
  A(\fl{g},-1,\tau-\delta) + A(\fl{g},\tau+\delta,1) + A(\fl{h},\mu,\delta) + \mathcal{E}_1,
\label{1stAppr}
\end{equation}
where
\begin{equation}
\renewcommand{\arraystretch}{1.4}
\begin{array}{rcl}
  \ds|\mathcal{E}_1| \htbbs&\leqx&\htbbs E_A(\fl{g},-1,\tau-\delta) + E_A(\fl{g},\tau+\delta,1) +
                                         E_A(\fl{h},\mu,\delta)\\
  \ds\htbbs&+&\htbbs 16\hts\eps D_{\htbs f} \log(\mu^{-1}) + 2\mu D_{1,f} .
\label{Appr1Err}
\end{array}
\end{equation}
Because we do not design our method to be used with any particular adaptive
quadrature, there is no way to know the values of the errors $E_A(\cdot,\cdot,\cdot)$
in advance. However, we assume that these errors will be estimated accurately enough
by the adaptive algorithm itself at the time the computations are being performed.

The best value of the parameter $\mu$ in (\ref{Appr1Err}) may be determined
in several ways. One could wish to minimize the error
\begin{equation}
  \mathcal{E}_{\mu} := 16 \hts\eps D_{\htbs f} \log(\mu^{-1}) + 2 \mu D_{1,f} .
\label{muErr}
\end{equation}
Simple computations show that $\mathcal{E}_{\mu}$ reaches its minimum at
$\mu = 8\hts\eps D_{\htbs f} D_{1,f}^{-1}$ (provided that $D_{1,f}\neq 0$).
From the practical point of view, however, this value of $\mu$ is usually
too small. The other reasonable, and our preferred way to determine
the proper value of $\mu$ is to split the error more evenly between
the last two terms of (\ref{Appr1Err}). In our implementation, we
determine the value of $\mu$ as the solution of the equation
\begin{equation}
  16 \hts\eps\log(\mu^{-1}) = 2 \mu .
\label{muCond}
\end{equation}
The above equation is satisfied for $\mu = 8\hts\eps W(\frac{1}{8}\eps^{-1})$, where
$W$ is the Lambert $W$-function (see, e.g., \cite{LambertW}). Subroutines for numerical
evaluation of the values of the Lambert $W$-function can be found in almost every system
for scientific computations. If $\eps=2^{-d}$, then $8\hts\eps W(\frac{1}{8}\eps^{-1})
\simeq (5.555d-9.747\log(d)-5.77)\eps$, with the relative error less than $0.14\%$,
for every $d\in\{20,21,\dots,2000\}$.

%%%%%%%%%%%%%%%%%%%%%%%%%%%%%%%%%%%%%%%%%%%%%%%%%%%%%%%%%%%%%%%%%%%%%%%%%%%%%%%%
\subsection{The second approach -- open-type quadratures} \label{SecErrOpen}

If, for approximating the two integrals on the right hand side of (\ref{MainFormula}), we use
an adaptive quadrature of the open type, then we are guaranteed that $0$ does not belong to
the set of nodes of the quadrature $A(h,0,\delta)$. Observe that, instead
of (\ref{1stAppr}), we may write
\begin{equation}
  I_{\tau}(f) \hts=\hts f(\tau)\,\log\frac{1-\tau}{1+\tau} +
  A(\fl{g},-1,\tau-\delta) + A(\fl{g},\tau+\delta,1) + A(\fl{h},0,\delta) \hts + \hts \mathcal{E}_2\hts ,
\label{OpenAppr}
\vspace{-0.5ex}
\end{equation}
where
\begin{equation}
\renewcommand{\arraystretch}{1.4}
\begin{array}{rcl}
  \ds|\mathcal{E}_2| \htbbs&\leq&\htbbs E_A(g,-1,\tau-\delta) + E_A(g,\tau+\delta,1) + E_A(h,0,\delta)\\
  \ds\htbbs&+&\htbbs \big| \tilde{A}(\fl{g}\htbs-\htbs g,-1,\tau-\delta) +
                     \tilde{A}(\fl{g}\htbs-\htbs g,\tau+\delta,1) +
                     \tilde{A}(\fl{h}\htbs-\htbs h,0,\delta) \big|.
\end{array}
\label{E2Def}
\end{equation}
In the above inequality, the symbols $\tilde{A}(\cdot,\cdot,\cdot)$ denote
the values of some (in general, arbitrary) linear quad\-ratures of the form (\ref{ASumSum}).
For example, $\tilde{A}(\fl{g}\htbs-\htbs g,-1,\tau-\delta)$ is the linear combination
of the form (\ref{ASumSum}), where the subinterval endpoints $s_k$ were
determined automatically when running the adaptive quadrature
$A(g,-1,\tau-\delta)$. Now, setting
\begin{equation}
    \mathcal{E}_o \hts=\hts
      \tilde{A}(\fl{g}\htbs-\htbs g,-1,\tau-\delta) +
      \tilde{A}(\fl{g}\htbs-\htbs g,\tau+\delta,1) +
      \tilde{A}(\fl{h}\htbs-\htbs h,0,\delta) ,\,\,
\label{EoDef}
\end{equation}
analogously to (\ref{CutEst}), we may obtain\vspace{-1ex}
\begin{equation}
   |\mathcal{E}_o| \,\leqx\, 16\hts\eps D_{\htbs f}\hts \tilde{A}(\frac{1}{x},0,1) .
\label{OpenEst}
\vspace{-0.5ex}
\end{equation}
The term $\tilde{A}(x^{-1},0,1)$ is not as pessimistic as the last integral in (\ref{CutEst}),
as it is always finite. As we shall show, the value $\tilde{A}(x^{-1},0,1)$ is not much
larger than $\log(x_{00}^{-1})$, where $x_{00}$ is the smallest node of
$\tilde{A}(x^{-1},0,1)$.

Without loss of generality, we may assume that the quadrature rule $Q$ defined
in (\ref{QDef}) satisfy the obvious condition: $\sum_{j=0}^{m} B_j = 1$.

\begin{lm} \label{l3}
Assume that a quadrature $Q$ defined in (\ref{QDef}) is of the open type
and let $Q(\ff,a,b)$ denote the quadrature $Q$ applied (after linear
transformation) to the interval $[a,b]$, i.e.
\begin{equation*}
  Q(\ff,a,b) = (b-a)\sum\limits_{j=0}^{m} B_j\hts\ff\big(a+(b-a)x_j\big)
   \quad (0<x_0<x_1<\dots<x_m<1) .
\end{equation*}
Then, for every $0\leq a < b$ and $c > 0$
\begin{equation}
  Q(\frac{1}{x},a,b) = Q(\frac{1}{x},ca,cb) ,
\label{QScaleEq}
\vspace{-0.25ex}
\end{equation}
and, if $a>0$,\vspace{-0.25ex}
\begin{equation*}
  \int_{a}^{b}\frac{1}{x}\hts dx = \int_{ca}^{cb}\frac{1}{x}\hts dx .
\vspace{0.5ex}
\end{equation*}
\end{lm}
\begin{proof}
The proof is straightforward.
\end{proof}

\begin{lm} \label{l4}
For every open type quadrature $Q$ of the form (\ref{QDef}) there exists a positive
constant $C_Q$ such that for all $\,0<a<b<1$
\begin{equation*}
  Q(\frac{1}{x},a,b)\, \leq\, C_Q \int_a^b \frac{1}{x}\hts dx .
\end{equation*}
\end{lm}
\begin{proof}
Let us define
\begin{equation*}
  R_Q(a,b) = \frac{Q(x^{-1},a,b)}{\int_a^b x^{-1} dx} .
\vspace{0.5ex}
\end{equation*}
From Lemma \ref{l3} we immediately obtain that
$R_Q(a,b) = R_Q(ab^{-1},1)$,
which further implies:
\begin{equation*}
  C_Q := \sup_{0<a<b<1}R_Q(a,b) = \sup_{0<a<1}R_Q(a,1).
\end{equation*}
As $\lim_{a\to 0}R_Q(a,1) = 0$ and $\lim_{a\to 1}R_Q(a,1) = 1$, we have $C_Q < \infty$.
\end{proof}

\begin{tw} For a linear open type quadrature $Q$ of the form (\ref{QDef}), let us define
\begin{equation*}
  D_Q = \hts\frac{Q(x^{-1},0,1)}{\log(x_0^{-1})} .
\vspace{-0.5ex}
\end{equation*}
If\vspace{-0.5ex}
\begin{equation*}
  \tilde{A}(\cdot,0,1) = \sum_{k=0}^{n-1} Q(\cdot,s_k,s_{k+1})
\end{equation*}
for some $0=s_0<s_1<\dots<s_n=1$ ($n\geq 1$), then
\begin{equation}
  \tilde{A}(\frac{1}{x},0,1) \,\leq\, \max\{C_Q,D_Q\}\log(x_{00}^{-1}) ,
\label{MainOpenEst}
\end{equation}
where $x_{00}$ is the smallest node of $Q(\cdot,s_0,s_1)$.
\end{tw}
\begin{proof}
From (\ref{QScaleEq}), Lemma \ref{l4}, the definition of $D_Q$,
and from $x_{00} = x_0 s_1 \leq x_0 < 1$, we obtain
\begin{align*}
  \tilde{A}(x^{-1},0,1) &\,=\, Q(x^{-1},0,s_1) + \sum_{k=1}^{n-1} Q(x^{-1},s_k,s_{k+1}) \\
    &\, \leq\,  Q(x^{-1},0,1) + C_Q\int_{s_1}^{1}x^{-1}dx
       \,=\,  D_Q\log(x_0^{-1}) + C_Q\log(s_1^{-1}) \\[0.75ex]
    &\,=\, D_Q\log(x_0^{-1}) + C_Q\big(\log(x_{00}^{-1}) - \log(x_0^{-1})\big) \\[1.75ex]
    &\,=\, (D_Q-C_Q)\log(x_0^{-1}) + C_Q\log(x_{00}^{-1})\\[1.75ex]
    &\,\leq\, \max\{C_Q,D_Q\}\log(x_{00}^{-1}) .
\vspace{-1ex}
\end{align*}
\end{proof}

\begin{wn}
By (\ref{EoDef})--(\ref{OpenEst}) and (\ref{MainOpenEst}), the inequality (\ref{E2Def})
may be rewritten as
\begin{equation}
  |\mathcal{E}_2| \,\leq\, E_A(g,-1,\tau-\delta) + E_A(g,\tau+\delta,1) + E_A(h,0,\delta)
    \,+\, |\mathcal{E}_o| ,
\label{FinalOpenEst}
\end{equation}
where
\begin{equation}
  |\mathcal{E}_o| \,\leqx\, 16\hts\eps D_{\htbs f}\max\{C_Q,D_Q\}\log(x_{00}^{-1}) .
\vspace{1ex}
\label{PesOpenEst}
\end{equation}
\end{wn}

For any given quadrature rule $Q$, the value of $D_Q$ can be immediately obtained
by its definition, but there seems to be no general way to compute
the constant $C_Q$. However, it can be found or estimated very easily
for most of important quadrature rules. If the error formula of
a quadrature $Q$ is of the form
\begin{equation*}
  \int_{0}^1 \ff(x) dx - Q(\ff,0,1) \,=\, \mathcal{W} \ff^{(2k)}(\xi_{\ff})
\end{equation*}
($k\in\mathbb{N}$, $0<\xi_{\ff}<1$, $\mathcal{W}>0$),
which is the case, e.g., for all Gauss-Legendre quadratures and
the open type Newton-Cotes rules (cf.\ \cite[\S2.6-\S2.7]{Davis}), then
we immediately obtain that $C_Q = 1$. For a one point rule
\begin{equation*}
  Q(\ff,0,1) = \ff(x_0) \qquad (0<x_0<1),
\end{equation*}
we have $C_Q = R_Q(a,1)$, where $a$ is a solution of the nonlinear
equation $(x_0-1)a^2+(1-2x_0)a+a\log(a)+x_0=0$. If a quadrature $Q$ of the
form~(\ref{QDef}) has positive coefficients and satisfies
\begin{equation*}
  \sum_{j=0}^{k-1} B_j < x_k < \sum_{j=0}^{k} B_j \qquad (0\leq k\leq m) ,
\end{equation*}
i.e.\ $Q$ is a linear combination of one point rules, then the value of $C_Q$ can
be estimated using the most pessimistic term. For example, for the $15$-point Kronrod
extension \cite{Kronrod} of the $7$-point Gauss-Legendre rule, we have
$C_Q<1.06<D_Q\simeq 1.29$.

%%%%%%%%%%%%%%%%%%%%%%%%%%%%%%%%%%%%%%%%%%%%%%%%%%%%%%%%%%%%%%%%%%%%%%%%%%%%%%%%
\subsubsection{The average case} \label{SecErrOpenAvg}

The error estimation (\ref{PesOpenEst}) describes the most
pessimistic case, where all round-off errors cumulate in the final
approximation. This is very unlikely, and therefore it seems advisable
to find some estimation of the arithmetic related errors in the average case.
The problem is that neither the values of the coefficients nor the location
of the adaptive quadrature nodes in (\ref{ASumSum}) are known in advance.
Thus, the precise probabilistic analysis could be extremely complicated.
However, one can try to apply a simplified model of the
influence of round-off errors.

Assume that all coefficients of the quadrature $\tilde{A}$ are equal, and that the
quadrature nodes are uniformly distributed, i.e 
\begin{equation}
  \tilde{A}(x^{-1},0,1) \,=\, \frac{1}{N-1}\sum_{i=1}^{N-1} \frac{N}{i}\hts,
\label{OpenNC}
\end{equation}
where $N$ is the total number of nodes. Observe that in such a case,
$\lim_{N\to\infty}\tilde{A}(x^{-1},0,1)/\log(N)=1$, which is consistent
with (\ref{MainOpenEst}), as for the quadrature (\ref{OpenNC}) we have $C_Q=1$
and $D_Q\to 1$ for $N\to\infty$. Let $U$ denote the continuous random variable
with the uniform distribution in the interval $[-1,1]$. Now, we treat
the error $\mathcal{E}_o$ defined in (\ref{EoDef}) as a random variable,
and -- instead of (\ref{OpenEst}) -- we obtain $\mathcal{E}_o \,=\, \eps D_{\htbs f} X$,
where $\eps D_{\htbs f}$ is the $f$-dependent value of the basic round-off error,
\begin{equation*}
  X \,=\, \frac{N}{N-1}\sum_{i=1}^{N-1} \frac{Y_i}{i} ,
\end{equation*}
is the random variable that corresponds to summing up the errors at all nodes,
and $Y_i$ are independent random variables corresponding to the cumulation of the round-off
errors at each single node. From (\ref{OpenEst}), we already know that $|Y_i|\leq16$ ($1\leq i< N$).
It seems reasonable to assume that for $i=1,2,\dots,N\htbs-1$ we have $Y_i=\sum_{k=1}^{16} U$.
We shall use the standard notation $\mathrm{E}(X)$ and $\mathrm{Var}(X)$ for the mean
value and variance of the random variable $X$. Obviously, $\mathrm{E}(X)=0$.
What we are interested in, however, is the value of $\mathrm{E}(|X|)$.
It can be readily verified that any random variable $X$ satisfies the inequality
$\mathrm{E}(|X-\mathrm{E}(X)|)\leq \sqrt{\mathrm{Var}(X)}$, which, in our case,
simplifies to $\mathrm{E}(|X|)\leq \sqrt{\mathrm{Var}(X)}$. As variances of independent
random variables add (cf.\ \cite[\S1.5]{Bill}), we immediately obtain
\begin{equation*}
  \mathrm{E}(|X|) \,\leq\, \sqrt{\frac{N^2}{(N-1)^2}\sum_{i=1}^{N-1} \frac{\mathrm{Var}(Y_i)}{i^2} } \,=\,
  \frac{4\hts N}{\sqrt{3}(N-1)}\sqrt{\sum_{i=1}^{N-1}\frac{1}{i^2}} \,=\,
  \frac{2\sqrt{2}\pi}{3} + O(N^{-1}) .
\end{equation*}
From the above inequality, it follows that the average error is bounded by the value
whose dependence on $N$ can be practically ignored. The asymptotic distribution of the random
variable $X$ can be given in an explicit form (see \cite{Mitra}) from which we may conclude
that the probability that $|X|>2\sqrt{2}\pi/3$ is (asymptotically) less than $1/3$.
This value is not yet satisfactory, however, by multiplying the right hand side
of the above inequality by $9/2$, we already obtain that
$\mathrm{Pr}(|X|>3\sqrt{2}\hts\pi) < 10^{-5}$, which turns
out to be quite sufficient from the practical point of view.

Generalizing the above result to the case of an arbitrary linear
quadrature $\tilde{A}$, we may expect that with a very high probability
\begin{equation}
  |\mathcal{E}_o| \,\leq\, 3\sqrt{2}\hts\pi\hts\eps D_{\htbs f}\max\{C_Q,D_Q\} .
\label{AvgOpenEst}
\vspace{0.5ex}
\end{equation}

In order to use the inequality (\ref{AvgOpenEst}), we have to assume that the
approximation to an integral computed by the adaptive quadrature is of the linear
form, i.e.\ of the form (\ref{ASumSum}), which is true for most of the adaptive
integration algorithms. Thus, (\ref{AvgOpenEst}) is a valid estimation
of the round-off error $|\mathcal{E}_o|$ in -- practically -- every case,
where an adaptive quadrature of the open type is used. However, computation
of the Cauchy principal value integral (\ref{CPInt1}) by the use of the
formula (\ref{OpenAppr}) also requires a proper error tolerance to be
set when calling the adaptive quadrature subroutine. To this end, we
can use the expression on the right hand side of (\ref{AvgOpenEst}),
but only in the case when the error estimator of the adaptive quadrature
is also of the linear type (is a linear combination of the values of integrand).
Otherwise, it may be necessary to apply the pessimistic formula, i.e.\ the
last term of~(\ref{FinalOpenEst}) (after assuming some reasonable
lower bound for the value of $x_{00}$).

%%%%%%%%%%%%%%%%%%%%%%%%%%%%%%%%%%%%%%%%%%%%%%%%%%%%%%%%%%%%%%%%%%%%%%%%%%%%%%%%
\subsection{Sensitivity to small changes of the parameter $\tau$}

When approximating the value of the integral (\ref{CPInt1})
numerically, we cannot in general assume that $\tau$ has an exact representation
in a given computer arithmetic. That means, that we usually approximate the integral
$I_{\tau(1+\al)}(f)$ instead of $I_{\tau}(f)$, for some small value of $|\al|\leq\eps$.
In some cases the values of these two integrals may be considerably different.

If $|\tau|\approx 1$, then the first term on the right hand side of (\ref{MainFormula})
has the largest influence on the value of the integral (\ref{CPInt1}). It is easily seen that
in this case even a small change of $\tau$ may significantly alter the value of the
integral. Let us denote $L(x) = \log((1-x)/(1+x))$. Then, we have
\begin{equation*}
  | L(\tau(1+\al)) - L(\tau) | \,\simeq\, |\al \tau L'(\tau)|
\end{equation*}
and, consequently,
\begin{equation}
 | I_{\tau(1+\al)}(f) - I_{\tau}(f) | \,\simeq\, \frac{2\hts|\al\tau f(\tau)|}{1-\tau^2}
   \,\leq\, \eps\frac{2\hts|\tau f(\tau)|}{1-\tau^2}.
\vspace{0.5ex}
\label{logErr}
\end{equation}

For the arbitrary value of $\tau\in(-1,1)$, assuming that $f(x+\al)\simeq f(x)$
for $-1<x<1$ and $|\al|\leq\eps$, we have
\begin{equation*}
  \cpint_{-1}^{1}\frac{f(x)}{x-\tau(1+\al)}\,dx \,\simeq\,
  \cpint_{-1-\al\tau}^{1-\al\tau}\frac{f(x)}{x-\tau}\,dx ,
\end{equation*}
which further implies
\begin{equation}
  | I_{\tau(1+\al)}(f) - I_{\tau}(f) |
  \,\simeq\, |\al\tau|\bigg|\frac{f(-1)}{1+\tau} + \frac{f(1)}{1-\tau}\bigg|
  \,\leq\, \eps|\tau|\bigg(\frac{|f(-1)|}{1+\tau} + \frac{|f(1)|}{1-\tau}\bigg) .
\label{tauErr}
\end{equation}
Observe that if $|\tau|\to 1$, then the ratio between the expression on the right hand side
of (\ref{logErr}) and the one of (\ref{tauErr}) approaches 1. In our implementation
we choose the larger of the two above estimates.

There is one more situation, when a small perturbation of $\tau$ may have quite
significant influence on the accuracy of the numerical approximation of the
integral (\ref{CPInt1}). If $|f''(\tau)|$ is very large, then $f'(\tau)$ changes
very rapidly near $\tau$, and a small perturbation of $\tau$ may considerably change
the values of the function $h$ (defined in (\ref{ghDef})) near $0$. Unfortunately,
as the function $f$ is unknown in advance, it seems impossible to find a reasonable
theoretical estimation of the influence of the value $|f''(\tau)|$ on the global error
$|I_{\tau(1+\al)}(f) - I_{\tau}(f)|$. However, we have confirmed experimentally,
that if $|\al|\leq\eps$, then the expression
\begin{equation}
  S\hts \eps\hts |\tau| \sqrt{|f''(\tau)|} ,
\label{der2Err}
\end{equation}
for $S\simeq 10$ is a good estimation of this influence. Note that the existence of
the second derivative $f''(\tau)$ is not required in practice, as we use
the second order divided difference with a very small step instead.

Assume that the function $f$ in (\ref{CPInt1}) is of the form $f(x) = f_c(x) = F(x+c)$ for
some (not large) value of $|c|$, and that $c$ has an inaccurate representation $c\hts(1+\al)$
($|\al|\leq\eps$) in the computer arithmetic. If $f_c(x)\simeq f_c(x+\al c)$ for every
$-1\leq x\leq1$, then we have $I_{\tau}(f_{c(1+\al)})\simeq I_{\tau-\al c}(f_c)$,
and therefore, from the practical point of view, it seems much more reasonable
to consider absolute distortions of the parameter $\tau$ rather than relative.
For the reason described above, in our implementation, we replace the factor
$|\tau|$ in (\ref{logErr}), (\ref{tauErr}), and (\ref{der2Err}) with 1.

%%%%%%%%%%%%%%%%%%%%%%%%%%%%%%%%%%%%%%%%%%%%%%%%%%%%%%%%%%%%%%%%%%%%%%%%%%%%%%%%
%%%%%%%%%%%%%%%%%%%%%%%%%%%%%%%%%%%%%%%%%%%%%%%%%%%%%%%%%%%%%%%%%%%%%%%%%%%%%%%%

%%%%%%%%%%%%%%%%%%%%%%%%%%%%%%%%%%%%%%%%%%%%%%%%%%%%%%%%%%%%%%%%%%%%%%%%%%%%%%%%
%%%%%%%%%%%%%%%%%%%%%%%%%%%%%%%%%%%%%%%%%%%%%%%%%%%%%%%%%%%%%%%%%%%%%%%%%%%%%%%%
\section{Numerical experiments} \label{SecNum}
\setcounter{equation}{0}
%%%%%%%%%%%%%%%%%%%%%%%%%%%%%%%%%%%%%%%%%%%%%%%%%%%%%%%%%%%%%%%%%%%%%%%%%%%%%%%%
%%%%%%%%%%%%%%%%%%%%%%%%%%%%%%%%%%%%%%%%%%%%%%%%%%%%%%%%%%%%%%%%%%%%%%%%%%%%%%%%

We have tested the proposed method for computing Cauchy principal value
integrals (\ref{CPInt1}) in the two popular systems for scientific computations,
{\Maple} and {\Matlab}, using adaptive quadratures available for these systems.

In order to implement the proposed method, a few small remaining gaps have
to be filled. First, we have to decide the values of the constants $K_0$ and $K_1$
in (\ref{DfDef}). Theoretically, there exists no perfect choice, as we do not
know how stable the computation of the values of the function $f$ is. In general,
assuming some reasonable level of stability of the integrand, it is sufficient
to set $K_0=1$ and $K_1=1$,
i.e.\
\begin{equation}
    D_{\htbs f} \,=\,
    \bigg(1+\frac{K_1}{2}\bigg)D_{1,f} + \frac{K_0}{2}D_{0,f} \,=\,
    \frac{3}{2}D_{1,f} + \frac{1}{2}D_{0,f} .
\label{DfSimple}
\end{equation}
We shall make one exception from the above rule for the reasons
explained in Section \ref{SecNumMatlab}.

Another task is to find a simple and effective way of approximating the values
$D_{1,f}$ and $D_{0,f}$ in (\ref{D1Def}) and (\ref{D0Def}). Although computing a good
approximation of $D_{0,f}$ is relatively simple, the computation of $D_{1,f}$ is already
quite difficult and time consuming. On the other hand, from (\ref{hErrEstPrec}) and
(\ref{gErrEstPrec}) we can see that only the values of $f'(x)$ and $\hat{f}(\tau,x)$
for $x$'s close to $\tau$ really matter, as the largest round-off errors appear
in a close neighbourhood of the singularity. Indeed, even the most naive
approximations, $D_{0,f}:=|f(\tau)|$ and $D_{1,f}:=|f'(\tau)|$, lead
to very satisfying practical results. In our implementation of the
proposed method, we use the following substitutes of
(\ref{D1Def}) and (\ref{D0Def}):\vspace{-1.5ex}
\begin{equation}
\begin{array}{l}
  D_{0,f} :=\hts \hat{D}_{0,f}(\tau) =\hts |f(\tau)|\hts ,\\[1ex]
  D_{1,f} :=\hts \hat{D}_{1,f}(\tau) =\hts  \max\bigg\{\big|f'(\tau)\big|;\,
    \omega_i\bigg|\frac{f(\tau\pm \theta_i)-f(\tau)}{\theta_i}\bigg|,
    i=1,2,\dots,j\bigg\},
\end{array}
\label{D01fSimple}
\end{equation}
where $j=1$, $\theta_1=\mu$, $\omega_1=1$, if we use (\ref{1stAppr})--(\ref{Appr1Err}), and
$j=4$, $\theta_{1,2,3,4}=1/41,1/35,1/16,1/11$, $\omega_{1,2,3,4}=2/3,4/7,1/2,1/3$,
if we use (\ref{OpenAppr})--(\ref{EoDef}), and (\ref{PesOpenEst}) or (\ref{AvgOpenEst}).
The value of $f'(\tau)$ is approximated by the second order symmetric divided difference.

The experiments were performed in the $64$-bit versions of \emph{Maple 16}
and \emph{Matlab R2014b}, on the computer powered by the $4$-core processor running
at $3.4$GHz. The subroutines for numerical evaluation of the Cauchy principal value
integrals, which were used to compute the presented numerical results are available
online at \url{http://www.ii.uni.wroc.pl/~pkl/programs/}.

%%%%%%%%%%%%%%%%%%%%%%%%%%%%%%%%%%%%%%%%%%%%%%%%%%%%%%%%%%%%%%%%%%%%%%%%%%%%%%%%
\subsection{Maple} \label{SecNumMaple}

The {\Maple} system is partially capable of computing Cauchy principal value integrals.
According to the system documentation, {\Maple} will be able to evaluate the integral (\ref{CPInt}),
if using some analytical transformation\footnote{The ability of symbolic computations is
{\Maple}'s main feature.} and the symmetry property similar to (\ref{Formula1}), the
integral can be transformed into a proper one. For example, {\Maple} can easily
compute the value of the integral (\ref{CPInt}) for $f(x) = \e^{\hts x}$, but
is unable to evaluate this integral for $f(x) = \e^{-x^2}$, whereas both integrals
are of the same difficulty from the numerical point of view. In other words,
{\Maple} is capable of evaluating the Cauchy principal value
integrals for which analytical expressions are known.

For numerical computations, {\Maple} uses non-standard floating point arithmetic.
The precision of computations may be modified at any time by defining the number of significant
decimal digits that are stored during the computation process. Moreover, many built-in {\Maple}
functions are allowed to temporarily increase the computation accuracy in order to attempt
to evaluate the final result with the relative error less than \tfl{5}{-d},
where $d$ is the requested number of decimal digits.

Following the {\Maple} philosophy, we wrote our subroutine for computing
integrals of the form (\ref{CPInt1}) in such a way that the final result is computed
up to $d$ significant decimal digits, and the relative approximation error is less
than\footnote{In fact, in our implementation, we approximate the integral
with the relative error less than $10^{-d-1}$, and then round the result to
$d$ significant decimal digits. In such a case, in theory, the final relative
error is only less than \tfl{5.1}{-d}.} \tfl{5}{-d}.
Here, it should be noted that in order to achieve an arbitrary accuracy
of the approximation to the integral (\ref{CPInt1}), the parameter $\tau$ has
to be given exactly. This is a serious restriction in general. However, in {\Maple}
all decimal numbers are represented exactly, as also all rational ones given
as a ratio of two integer numbers. In {\Maple}, it is also possible to
define $\tau$ as an unevaluated expression, which will be computed with
a desired accuracy during the computation of the integral (\ref{CPInt1}).

As we could find no detailed information on the {\Maple} built-in adaptive
integration subroutine, we have decided to use the approach described in
Section \ref{SecErrCut}, i.e.\ to exclude a very small neighbourhood of the
singularity point from the interval of integration (see (\ref{1stAppr})--(\ref{Appr1Err})).
In the case of Cauchy principal value integrals, there seems to be no way to estimate
the relative approximation error, unless we know the value of the integral itself.
Therefore, in the \emph{Maple }implementation of the proposed method, a rough
approximation to the integral (\ref{CPInt1}) is computed first. Then, we use
(\ref{muErr}), (\ref{muCond}), (\ref{logErr}), and (\ref{der2Err}) to find
the required precision~$\eps$ for the computation of the integral (\ref{CPInt1}),
so that the final relative approximation error is less than the expected
tolerance\footnote{It should be noted, that due to the {\Maple}
adjustable machine precision, a very simple solution, like e.g.,
$\eps = 10^{-K}$ for $K\geq\log_{10}(\mu^{-2})$, and $\mu=\max\{1,|f'(\tau)|\}^{-1}10^{-d}/4$,
would produce accurate results. In our implementation, the above values are only determined
in a close to the optimal way, which helps in saving some computation time.}.
The possible loss of significant digits that may occur when adding up
the three terms in (\ref{MainFormula}) is also taken care for.

The results of the experiments are gathered in Tables \ref{Tab1}--\ref{Tab4}.
We computed several integrals of the form (\ref{CPInt1}) for
$f\equiv f_i$ ($i=1,2,3,4$), where\vspace{-1ex}
\begin{equation}
\renewcommand{\arraystretch}{1.4}
\begin{array}{l}
f_1(x) = \e^{\hts 4x},\\
f_2(x) = \sinh(x)\cos(3193x),\\
f_3(x) = \e^{-(x+0.5)^2},\\
f_4(x) = \sin\big(8x+\e^{\hts 8x}\big),
\end{array}
\label{f1-4Def}
\end{equation}
and $\tau \in \{-0.7,\,0.11,\,0.55,\,0.99\}$.
For the functions $f_1$ and $f_2$, the accurate value of the integral (\ref{CPInt1})
can be computed analytically. In the case of the two other integrands, we assumed
as accurate the $100$-digit approximations computed using the presented method.
In our experiment, the value of each integral was numerically approximated
up to $d$ significant decimal digits for $d \in \{32,48,64\}$, and
then compared to the accurate result. As we can see, no relative error
exceeded \tfl{5}{-d}. We have performed several hundreds
more experiments of the similar type with a~similar outcome.

\vspace*{-0.25ex}
\begin{table}[!ht]
\begin{center}
\caption{\small
The requested numbers of significant decimal digits and the actual relative
errors of the approximations to the integral (\ref{CPInt1}) for
$f(x) = \e^{\hts 4x}$,
computed in {\Maple} using the method proposed in this paper.
}\label{Tab1}
\renewcommand{\arraystretch}{1.15}
\setlength\tabcolsep{1.5ex}
\vspace{0.9ex}
\begin{tabular}{C{24ex}C{14.875ex}C{14.875ex}C{14.875ex}C{14.875ex}}\hline
${}_{\ds\mathrm{requested~number~of}}$ & \multicolumn{4}{c}{relative approximation error}\\\cline{2-5}
${}^{\ds\mathrm{decimal~digits}}$ & \small{$\tau=-0.7$} & \small{$\tau=0.11$} &
                                    \small{$\tau=0.55$} & \small{$\tau=0.99$}\\\hline
%$16$ & \tfl{2.7}{-17} & \tfl{1.5}{-16} & \tfl{2.5}{-17} & \tfl{2.3}{-16}\\
$32$ & \tfl{3.6}{-33} & \tfl{1.1}{-32} & \tfl{5.9}{-33} & \tfl{5.1}{-35}\\
$48$ & \tfl{4.0}{-49} & \tfl{7.0}{-49} & \tfl{1.3}{-48} & \tfl{5.8}{-49}\\
$64$ & \tfl{3.9}{-65} & \tfl{2.2}{-64} & \tfl{6.0}{-66} & \tfl{3.0}{-64}\\
%$80$ & \tfl{8.5}{-82} & \tfl{1.7}{-80} & \tfl{1.0}{-80} & \tfl{1.1}{-80}\\
\hline
\end{tabular}
\vspace{-2.5ex}
\end{center}
\end{table}

\begin{table}[!ht]
\begin{center}
\caption{\small
The requested numbers of significant decimal digits and the actual relative
errors of the approximations to the integral (\ref{CPInt1}) for
$f(x) = \sinh(x)\cos(3193x)$,
computed in {\Maple} using the method proposed in this paper.
}\label{Tab2}
\renewcommand{\arraystretch}{1.15}
\setlength\tabcolsep{1.5ex}
\vspace{0.9ex}
\begin{tabular}{C{24ex}C{14.875ex}C{14.875ex}C{14.875ex}C{14.875ex}}\hline
${}_{\ds\mathrm{requested~number~of}}$ & \multicolumn{4}{c}{relative approximation error}\\\cline{2-5}
${}^{\ds\mathrm{decimal~digits}}$ & \small{$\tau=-0.7$} & \small{$\tau=0.11$} &
                                    \small{$\tau=0.55$} & \small{$\tau=0.99$}\\\hline
%$16$ & \tfl{1.2}{-16} & \tfl{6.1}{-18} & \tfl{1.0}{-16} & \tfl{1.2}{-16}\\
$32$ & \tfl{1.6}{-32} & \tfl{5.3}{-33} & \tfl{8.9}{-33} & \tfl{1.9}{-32}\\
$48$ & \tfl{1.4}{-48} & \tfl{2.3}{-48} & \tfl{3.3}{-49} & \tfl{9.2}{-49}\\
$64$ & \tfl{1.3}{-64} & \tfl{9.7}{-65} & \tfl{3.1}{-65} & \tfl{1.3}{-64}\\
%$80$ & \tfl{1.8}{-80} & \tfl{6.3}{-81} & \tfl{4.0}{-81} & \tfl{5.3}{-81}\\
\hline
\end{tabular}
\vspace{-2.5ex}
\end{center}
\end{table}

\begin{table}[!ht]
\begin{center}
\caption{\small
The requested numbers of significant decimal digits and the actual relative
errors of the approximations to the integral (\ref{CPInt1}) for
$f(x) = \e^{-(x+0.5)^2}$,
computed in {\Maple} using the method proposed in this paper.
}\label{Tab3}
\renewcommand{\arraystretch}{1.15}
\setlength\tabcolsep{1.5ex}
\vspace{0.9ex}
\begin{tabular}{C{24ex}C{14.875ex}C{14.875ex}C{14.875ex}C{14.875ex}}\hline
${}_{\ds\mathrm{requested~number~of}}$ & \multicolumn{4}{c}{relative approximation error}\\\cline{2-5}
${}^{\ds\mathrm{decimal~digits}}$ & \small{$\tau=-0.7$} & \small{$\tau=0.11$} &
                                    \small{$\tau=0.55$} & \small{$\tau=0.99$}\\\hline
%$16$ & \tfl{6.0}{-17} & \tfl{1.6}{-16} & \tfl{2.7}{-17} & \tfl{5.8}{-17}\\
$32$ & \tfl{2.6}{-32} & \tfl{8.6}{-33} & \tfl{1.2}{-34} & \tfl{8.8}{-33}\\
$48$ & \tfl{8.0}{-49} & \tfl{1.2}{-48} & \tfl{5.2}{-49} & \tfl{2.3}{-49}\\
$64$ & \tfl{3.1}{-64} & \tfl{4.0}{-65} & \tfl{1.6}{-65} & \tfl{3.9}{-65}\\
%$80$ & \tfl{2.1}{-81} & \tfl{1.5}{-80} & \tfl{3.0}{-81} & \tfl{4.5}{-81}\\
\hline
\end{tabular}
\vspace{-2.5ex}
\end{center}
\end{table}

\begin{table}[!ht]
\begin{center}
\caption{\small
The requested numbers of significant decimal digits and the actual relative
errors of the approximations to the integral (\ref{CPInt1}) for
$f(x) = \sin\big(8x+\e^{\hts 8x}\big)$,
computed in {\Maple} using the method proposed in this paper.
}\label{Tab4}
\renewcommand{\arraystretch}{1.15}
\setlength\tabcolsep{1.5ex}
\vspace{0.9ex}
\begin{tabular}{C{24ex}C{14.875ex}C{14.875ex}C{14.875ex}C{14.875ex}}\hline
${}_{\ds\mathrm{requested~number~of}}$ & \multicolumn{4}{c}{relative approximation error}\\\cline{2-5}
${}^{\ds\mathrm{decimal~digits}}$ & \small{$\tau=-0.7$} & \small{$\tau=0.11$} &
                                    \small{$\tau=0.55$} & \small{$\tau=0.99$}\\\hline
%$16$ & \tfl{7.9}{-17} & \tfl{8.4}{-17} & \tfl{2.8}{-16} & \tfl{7.9}{-17}\\
$32$ & \tfl{1.7}{-32} & \tfl{3.9}{-34} & \tfl{3.1}{-32} & \tfl{3.0}{-33}\\
$48$ & \tfl{2.3}{-49} & \tfl{4.1}{-49} & \tfl{1.8}{-48} & \tfl{1.9}{-49}\\
$64$ & \tfl{1.5}{-64} & \tfl{1.4}{-64} & \tfl{8.3}{-65} & \tfl{5.2}{-65}\\
%$80$ & \tfl{6.5}{-81} & \tfl{1.6}{-80} & \tfl{4.7}{-82} & \tfl{6.2}{-81}\\
\hline
\end{tabular}
\vspace{-2ex}
\end{center}
\end{table}

%%%%%%%%%%%%%%%%%%%%%%%%%%%%%%%%%%%%%%%%%%%%%%%%%%%%%%%%%%%%%%%%%%%%%%%%%%%%%%%%
\subsection{Matlab} \label{SecNumMatlab}

We have performed much more extensive numerical tests of the proposed method in
{\Matlab}. The {\Matlab} system uses the standard double precision floating point
arithmetic with machine epsilon $\eps = 2^{-52} \simeq \mfl{2.22}{-16}$. Our goal
now, as we cannot manipulate the machine precision, is to compute the possibly most
accurate approximation to the integral (\ref{CPInt1}). We also wish to compute
a reliable estimation of the approximation error, a very valuable
information for the user about the accuracy of the approximation.

In the first set of experiments we have used the main\footnote{In the newest versions
of the {\Matlab} system, the main adaptive integration subroutine is called {\tt integral}.
Currently, it uses exactly the same algorithm (based on \cite{Shampine}) as {\quadgk},
the latter, however, offers more user control, and, what is the most important,
reports the estimated bound on the absolute error.} {\Matlab} adaptive quadrature,
{\tt quadgk}. It is based on the idea presented in \cite{Shampine}. The {\quadgk}
algorithm uses the most popular ($15$,$7$) Gauss-Kronrod pair as the basic quadrature
rule. This is the open type quadrature, and thus we can use the formula (\ref{OpenAppr})
for approximating the value of the integral (\ref{CPInt1}). In order to estimate the errors
$E_A(\cdot,\cdot,\cdot)$ in (\ref{E2Def}), we use the approximated error bounds reported
by the {\quadgk} function. As both the basic quadrature rule and the error
estimator in {\quadgk} are of the linear type, we use
(\ref{EoDef}) and (\ref{AvgOpenEst}) with%
\footnote{The {\quadgk} algorithm includes a nonlinear transformation
of intervals of integration in order to annihilate potential weak endpoint
singularities of the integrand (see \cite{Shampine} for more details).
This procedure introduces some additional perturbations of the
integrand and its argument. For that reason, in this particular
case, we increase the values of the constants $K_0$
and $K_1$ in (\ref{DfSimple}) to $K_0=K_1=2$.}
$D_{\htbs f} := 2 \hat{D}_{1,f}(\tau) + \hat{D}_{0,f}(\tau)$,
for estimating the influence of the round-off errors. To estimate
the errors related to possible perturbations of~$\tau$, we apply (\ref{logErr})--(\ref{der2Err}).
The error tolerance for the adaptive quadrature {\quadgk} is automatically controlled
and set to the largest of the four above estimates
(i.e.\ (\ref{AvgOpenEst})--(\ref{der2Err})).

The proposed method for computing Cauchy principal value integrals is compared to the
algorithm presented by Hasegawa and Torii in \cite{HTCauchy}, which is
a very good automatic quadrature scheme for evaluating such type of integrals.
In the algorithm of \cite{HTCauchy}, the function $f$ is interpolated at the
automatically selected number $N$ of Chebyshev points $t_k=\cos(k\pi N^{-1})$ ($k=0,1,\dots,N$).
Then, a Chebyshev series expansion to the last integral in (\ref{Formula2}) is computed by means of
a three-term recurrence relation. The effectiveness of this algorithm strongly depends on the
smoothness of the function $f$. In our implementation of the method of \cite{HTCauchy}, we
determine the optimal value of $N$ basing on the magnitude and the decay rate of the Chebyshev
coefficients of the integrand, in a little different way\footnote{The error estimate presented
in \cite{HTCauchy} does not include the influence of round-off errors and arithmetic
limitations, and may fail in case of smaller error tolerances. Therefore, we had to
find a way to determine which coefficients in the Chebyshev expansion of the
function $f$ may be considerably affected by the round-off errors, adjust
the error tolerance accordingly during the computations, and finally,
remove the probably too inaccurate terms from the above-mentioned
expansion.} compared to the results of \cite{HTCauchy}. The estimation
of the absolute error is computed according to the formulae given in \cite{HTCauchy}.

In Tables \ref{Tab5}-\ref{Tab9} we present the results of
numerical experiments performed for the functions $f_1$, $f_2$ (cf.\ \ref{f1-4Def}),
\begin{equation}
\renewcommand{\arraystretch}{1.5}
\begin{array}{l}
f_5(x) = \frac{1}{100}(x-1.00001)^{-2},\\
f_6(x) = \sqrt{|\cos(44x)|^3},\\
f_7(x) = \sin(\sqrt{1+x})\log(1-x),
\end{array}
\label{f5-7Def}
\end{equation}
and a few selected values of $\tau$. As the accurate values, we have assumed
the 32-decimal-digit results computed in {\Maple}. We compare the absolute errors
of the obtained approximations, the computed error estimations, and the computation
times\footnote{The {\quadgk} adaptive quadrature is implemented in the {\Matlab} system
as a standard \mbox{m-file} (with no hardware optimisation). The high efficiency of the quadrature
was achieved by using fast vectorised operations (the set of integrals over the number of
subintervals is computed at the same time). In our implementation of
the algorithm of \cite{HTCauchy}, we have used the {\Matlab} built-in, hardware
optimised subroutine for computing the Fast Fourier Transform, the most costly
part of the algorithm. Thus, both compared method may be assumed equally optimised,
with a small advantage towards the method of \cite{HTCauchy}.}.
This is no surprise that for the functions $f_1$ and $f_2$ which are analytic and
relatively \emph{nice}, the algorithm of Hasegawa and Torii computes the results faster
than the proposed method. This is typical for adaptive quadratures to work slower than
dedicated methods in the case of regular functions. Only for complicated or not very
smooth integrands, adaptive quadratures show their potential. Here, the situation is
very similar. In case of the functions $f_5$ and $f_6$, the adaptive approach turned
out to be more efficient than its counterpart. There are no significant differences
between the tested methods in the accuracies of the computed approximations
to the integrals, maybe except for the function $f_5$ (Table~\ref{Tab7}).
As the algorithm of \cite{HTCauchy} features the uniform (with respect to $\tau$) error
estimation, the proposed method follows the behaviour of the integrand slightly better
when using (\ref{D01fSimple}). On the other hand, if several integrals with the same
function $f$ but different parameters $\tau$ are to be computed, then a lot of
computations can be saved in the algorithm of Hasegawa and Torii.

\begin{table}[!ht]
\begin{center}
\caption{\small
The comparison of the absolute approximation error ($E_{\tau}$), its estimation
($E^{\mathrm{est}}_{\tau}$), and the efficiency (computation time) of the proposed
method (based on the {\quadgk} adaptive quadrature) with the method of \cite{HTCauchy},
in the case of the integral (\ref{CPInt1}) with
$f(x)=\e^{\hts 4x}$.
The time unit $T =\hts 0.00021s$.
}\label{Tab5}
\renewcommand{\arraystretch}{1.2}
\setlength\tabcolsep{1.5ex}
\vspace{0.9ex}
\begin{tabular}{C{9ex}|C{12.5ex}C{12.5ex}C{9ex}|C{12.5ex}C{12.5ex}C{9ex}}\hline
       & \multicolumn{3}{c|}{The method of \cite{HTCauchy} }
       & \multicolumn{3}{c}{The proposed method}\\\cline{2-7}
${}^{{}^{\ds\tau}}$ & $E_{\tau}$ & [$\hts E^{\mathrm{est}}_{\tau}\hts$] & Time
                    & $E_{\tau}$ & [$\hts E^{\mathrm{est}}_{\tau}\hts$] & Time\\\hline
$-0.22\hpo\hpo$ &  \tfl{3.7}{-14}     &       \uar      &  \uar  & \tfl{1.8}{-15} & \tflb{6.2}{-14} & $4.7T$\\
$\hpm0.667\hpo$ &  \tfl{7.1}{-15}     & \tflb{1.8}{-13} & $1.0T$ & \tfl{7.1}{-15} & \tflb{6.8}{-13} & $4.7T$\\
   $\hpm0.9995$ & \tflu{\,6.1}{-12}\, &       \dar      &  \dar  & \tfl{6.1}{-12} & \tflb{2.1}{-11} & $4.7T$\\\hline
\end{tabular}
\vspace{-2.5ex}
\end{center}
\end{table}

\begin{table}[!ht]
\begin{center}
\caption{\small
The comparison of the absolute approximation error ($E_{\tau}$), its estimation
($E^{\mathrm{est}}_{\tau}$), and the efficiency (computation time) of the proposed
method (based on the {\quadgk} adaptive quadrature) with the method of \cite{HTCauchy},
in the case of the integral (\ref{CPInt1}) with
$f(x)=\sinh(x)\cos(3193x)$.
The time unit $T =\hts 0.0016s$.
}\label{Tab6}
\renewcommand{\arraystretch}{1.2}
\setlength\tabcolsep{1.5ex}
\vspace{0.9ex}
\begin{tabular}{C{9ex}|C{12.5ex}C{12.5ex}C{9ex}|C{12.5ex}C{12.5ex}C{9ex}}\hline
       & \multicolumn{3}{c|}{The method of \cite{HTCauchy} }
       & \multicolumn{3}{c}{The proposed method}\\\cline{2-7}
${}^{{}^{\ds\tau}}$ & $E_{\tau}$ & [$\hts E^{\mathrm{est}}_{\tau}\hts$] & Time
                    & $E_{\tau}$ & [$\hts E^{\mathrm{est}}_{\tau}\hts$] & Time\\\hline
$-0.22\hpo\hpo$ & \tfl{1.8}{-13} &       \uar       &  \uar  & \tfl{7.2}{-14} & [\tfl{7.0}{-12}] & $3.5T$\\
$\hpm0.667\hpo$ & \tfl{5.9}{-13} & [\tfl{2.6}{-12}] & $1.0T$ & \tfl{4.4}{-13} & [\tfl{1.1}{-11}] & $4.6T$\\
$\hpm0.906\hpo$ & \tfl{1.2}{-12} &       \dar       &  \dar  & \tfl{1.0}{-12} & [\tfl{3.0}{-11}] & $3.7T$\\\hline
\end{tabular}
\vspace{-2.5ex}
\end{center}
\end{table}

\begin{table}[!ht]
\begin{center}
\caption{\small
The comparison of the absolute approximation error ($E_{\tau}$), its estimation
($E^{\mathrm{est}}_{\tau}$), and the efficiency (computation time) of the proposed
method (based on the {\quadgk} adaptive quadrature) with the method of \cite{HTCauchy},
in the case of the integral (\ref{CPInt1}) with
$f(x) = \frac{1}{100}(x-1.00001)^{-2}$.
The time unit $T =\hts 0.0027s$.
}\label{Tab7}
\renewcommand{\arraystretch}{1.2}
\setlength\tabcolsep{1.5ex}
\vspace{0.9ex}
\begin{tabular}{C{9ex}|C{12.5ex}C{12.5ex}C{9ex}|C{12.5ex}C{12.5ex}C{9ex}}\hline
       & \multicolumn{3}{c|}{The method of \cite{HTCauchy} }
       & \multicolumn{3}{c}{The proposed method}\\\cline{2-7}
${}^{{}^{\ds\tau}}$ & $E_{\tau}$ & [$\hts E^{\mathrm{est}}_{\tau}\hts$] & Time
                    & $E_{\tau}$ & [$\hts E^{\mathrm{est}}_{\tau}\hts$] & Time\\\hline
$-0.22\hpo\hpo$ & \tfl{8.8}{-8} &      \uar       &  \uar  & \tfl{5.9}{-9} & [\tfl{1.9}{-8}] & $0.67T$\\
$\hpm0.667\hpo$ & \tfl{8.5}{-8} & [\tfl{6.1}{-6}] & $1.0T$ & \tfl{2.0}{-8} & [\tfl{5.1}{-8}] & $0.67T$\\
$\hpm0.906\hpo$ & \tfl{1.7}{-6} &      \dar       &  \dar  & \tfl{6.2}{-8} & [\tfl{2.0}{-7}] & $0.63T$\\\hline
\end{tabular}
\vspace{-2.5ex}
\end{center}
\end{table}

\begin{table}[!ht]
\begin{center}
\caption{\small
The comparison of the absolute approximation error ($E_{\tau}$), its estimation
($E^{\mathrm{est}}_{\tau}$), and the efficiency (computation time) of the proposed
method (based on the {\quadgk} adaptive quadrature) with the method of \cite{HTCauchy},
in the case of the integral (\ref{CPInt1}) with
$f(x)=\sqrt{|\cos(44x)|^3}$.
The time unit $T =\hts 1.1s$.
}\label{Tab8}
\renewcommand{\arraystretch}{1.2}
\setlength\tabcolsep{1.5ex}
\vspace{0.9ex}
\begin{tabular}{C{9ex}|C{12.5ex}C{12.5ex}C{9ex}|C{12.5ex}C{12.5ex}C{9ex}}\hline
       & \multicolumn{3}{c|}{The method of \cite{HTCauchy} }
       & \multicolumn{3}{c}{The proposed method}\\\cline{2-7}
${}^{{}^{\ds\tau}}$ & $E_{\tau}$ & [$\hts E^{\mathrm{est}}_{\tau}\hts$] & Time
                    & $E_{\tau}$ & [$\hts E^{\mathrm{est}}_{\tau}\hts$] & Time\\\hline
$-0.22\hpo\hpo$ & \tfl{5.2}{-12} &       \uar       &  \uar  & \tfl{8.2}{-15} & [\tfl{4.0}{-13}] & $0.0096T$\\
$\hpm0.667\hpo$ & \tfl{9.9}{-12} & [\tfl{1.2}{-11}] & $1.0T$ & \tfl{2.8}{-14} & [\tfl{5.8}{-13}] & $0.0083T$\\
$\hpm0.906\hpo$ & \tfl{1.6}{-12} &       \dar       &  \dar  & \tfl{1.6}{-14} & [\tfl{5.7}{-13}] & $0.0077T$\\\hline
\end{tabular}
\vspace{-2ex}
\end{center}
\end{table}

\begin{table}[!ht]
\begin{center}
\caption{\small
The comparison of the absolute approximation error ($E_{\tau}$), its estimation
($E^{\mathrm{est}}_{\tau}$), and the efficiency (computation time) of the proposed
method (based on the {\quadgk} adaptive quadrature) with the method of \cite{HTCauchy},
in the case of the integral (\ref{CPInt1}) with
$f(x)=\sin(\sqrt{1+x})\log(1-x)$.
The time unit $T \simeq\hts 0.003s$.
}\label{Tab9}
\renewcommand{\arraystretch}{1.18}
\setlength\tabcolsep{1.5ex}
\vspace{0.9ex}
\begin{tabular}{C{9ex}|C{12ex}C{12ex}C{9ex}|C{12ex}C{12ex}C{9ex}}\hline
       & \multicolumn{3}{c|}{The method of \cite{HTCauchy} }
       & \multicolumn{3}{c}{The proposed method}\\\cline{2-7}
${}^{{}^{\ds\tau}}$ & $E_{\tau}$ & [$\hts E^{\mathrm{est}}_{\tau}\hts$] & Time
                    & $E_{\tau}$ & [$\hts E^{\mathrm{est}}_{\tau}\hts$] & Time\\\hline
$\hpm0.667\hpo$ &       \na      &       \uar       &  \uar  & \tfl{1.8}{-15} & [\tfl{9.2}{-14}] & $1T$\\
$\hpm0.906\hpo$ &       \na      &        \na       &  \sna  & \tfl{5.7}{-15} & [\tfl{3.4}{-13}] & $1T$\\
$\hpm0.9995$    &       \na      &       \dar       &  \dar  & \tfl{8.0}{-13} & [\tfl{1.3}{-10}] & $1T$\\\hline
\end{tabular}
\vspace{-2ex}
\end{center}
\end{table}

In the last example presented in Table \ref{Tab5} we have set $\tau=0.9995$,
which caused the error estimation of the algorithm of \cite{HTCauchy} to fail.
The relatively large round-off error appeared during the evaluation of the one before
last term in (\ref{Formula2}). This suggests that the additional error estimates similar
to (\ref{logErr}) and (\ref{tauErr}) may be used to increase reliability of any algorithm
based on the formula similar to (\ref{Formula2}).

The function $\htbs f\equiv f_7$ is not bounded in $[-1,1]$, what is required in the
method of \cite{HTCauchy}. The first derivative of $f_7$ is also unbounded. However, $f{\hts'}$
needs to be bounded only in some neighbourhood of the singularity point~$\tau$ for
the present method to work properly in practice. The results
for $I_{\tau}(f_7)$ are listed in Table \ref{Tab9}.

Now, we shall concentrate only on the method which was proposed in this paper. A few values
reported in Tables \ref{Tab5}--\ref{Tab9} may not be a satisfactory proof of its reliability. First,
we shall verify whether the estimates derived in the paper are really necessary. To this end,
we have computed the integral (\ref{CPInt1}) for (a little irregularly behaving function)
$f\equiv f_8$, where
\begin{equation*}
  f_8(x) = \sin(33x) + \e^{\sin(\e^{4x})} ,
\end{equation*}
according to the formula (\ref{MainFormula}), for $19999$ different, equally distributed values
of $\tau$. Each of the two integrals in (\ref{MainFormula}) was computed using the {\quadgk}
function. The requested absolute error tolerance was manually set to $1000\hts\eps$ (note
that the ordinary integral $\int_{-1}^{1}f_8(x)dx$ is correctly computed by the {\quadgk}
algorithm with the absolute tolerance $5\eps$). As the error estimate, we have used
only the approximated error bounds reported by the {\quadgk} function. The graphs of
the actual absolute error and its estimation obtained in this way are presented in
Figure \ref{sinexp-gk0}. As we can see, for many values of $\tau$ ($3724$ of $19999$)
the error estimation failed, which confirms our suspicions that the round-off
errors that appear during the computation of the Cauchy principal value
integrals may not be properly estimated by adaptive quadrature itself.

\begin{figure}[!ht]
\centering\vspace*{-0.0ex}
\includegraphics[width=14cm,height=5.5cm,angle=0]{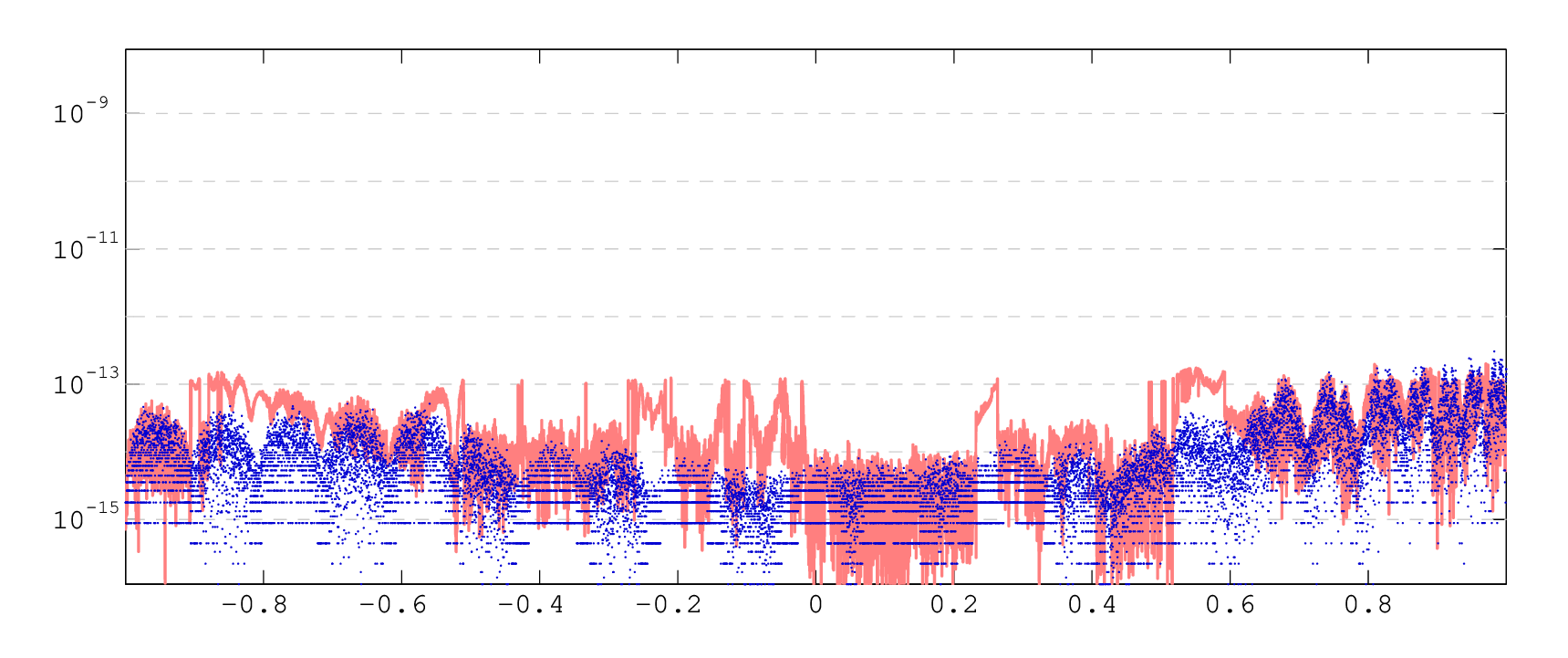}
\vspace{-3.5ex}
\caption{\small The graph of the absolute error (dark blue dots) of the approximation
to the integral (\ref{CPInt1}) for $f(x) = \sin(33x) + \exp(\sin(\e^{4x}))$, computed for
$19999$ equidistant values of $\tau\in(-1,1)$ using only (\ref{MainFormula}) and the adaptive
quadrature {\quadgk} from the {\Matlab} system, and the graph of the numerical
estimation of this error (irregular light red line) computed by the
adaptive quadrature itself.}
\label{sinexp-gk0}
\vspace*{2ex}
\end{figure}

Figure \ref{sinexp-gk} presents the results obtained for the
same integral, but with the tolerance control and the error estimates
(\ref{AvgOpenEst})--(\ref{der2Err}) turned back on. Now, the error is safely estimated,
the gap between the estimation line and the ''cloud'' of actual errors varies from about
$1$ decimal digit (for $\tau$'s near the endpoints) to about $2$ decimal digits
(for $\tau$'s near the centre of the interval). One may fault that the estimation is
not tight enough, if $\tau$ lies in the middle part of the interval. This is because
the computation of values of the function $f_8$ is, despite its complicated
formula, quite stable. Moreover, in this example, the round-off errors that appear
during the computation of values of the functions $g$ and $h$ in (\ref{ghDef})
depend more on $|\eps x|$ than on $\eps$, while the latter is assumed in our error
estimations. To illustrate it more clearly, in the next experiment, we compute
the integrals (\ref{CPInt1}) for $f \equiv f_9$, where
\begin{equation}
  f_9(x) = f_8(j(x)),\quad \mathrm{and}\quad j(x) = \arcsin(\sin(2\pi+x)) .
\label{f9Def}
\end{equation}
Clearly, in theory, $j(x) = x$ for every $x\in[-\pi,\pi]$,
and, consequently, $I_{\tau}(f_9) = I_{\tau}(f_8)$ for each $\tau \in (-1,1)$.
However, the computation of values of the function $f_9$ is much less
stable than in the case of $f_8$. The graph of the absolute error (and its estimation)
of the approximation to the integral (\ref{CPInt1}) with $f\equiv f_9$, computed
using the proposed method, can be found in Figure \ref{sinexp_4e-gk}. The error estimation
reflects the behaviour of the actual error very well. The error nowhere exceeds its
estimation, and the distance between them is greater than $0.1$ of
decimal digit for all values of~$\tau$.

\begin{figure}[!ht]
\centering\vspace*{-0.0ex}
\includegraphics[width=14cm,height=5.5cm,angle=0]{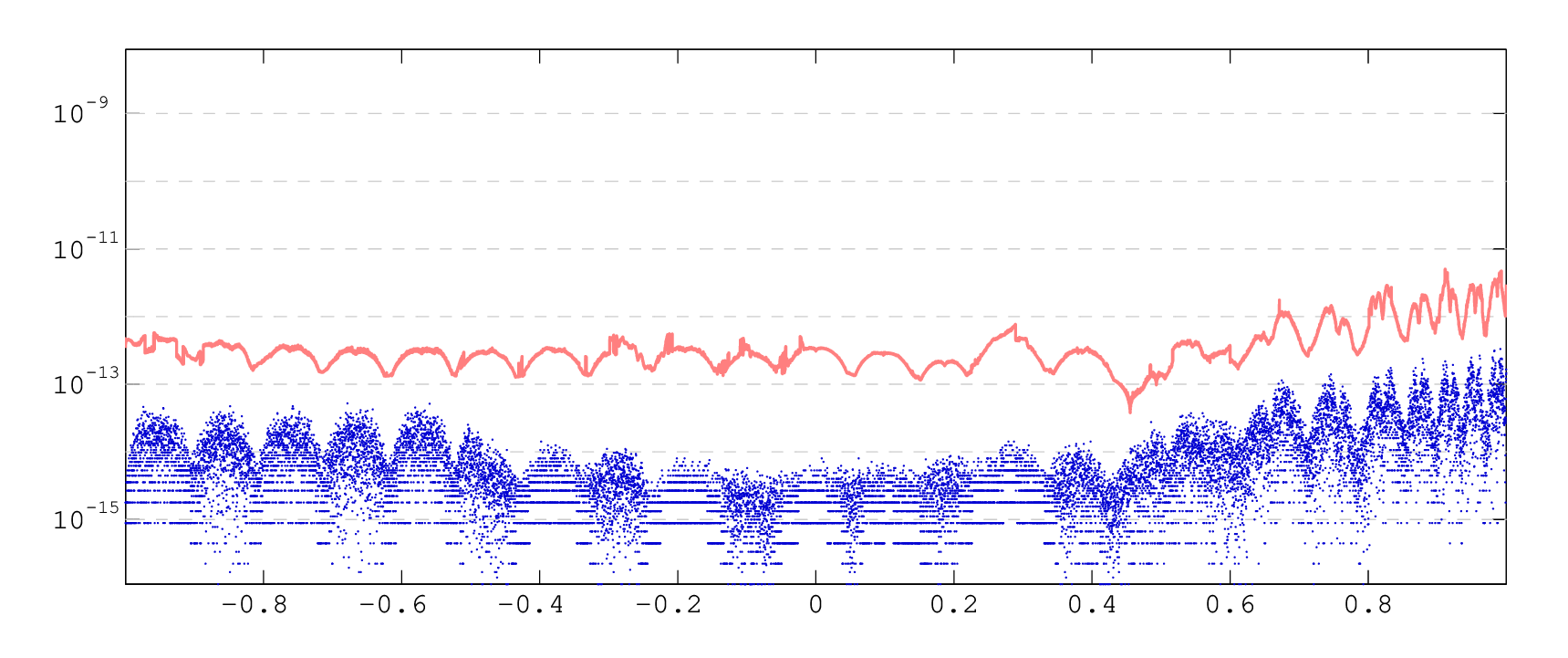}
\vspace{-3.5ex}
\caption{\small The graph of the absolute error (dark blue dots) of the approximation
to the integral (\ref{CPInt1}) for $f(x) = \sin(33x) + \exp(\sin(\e^{4x}))$, computed for
$19999$ equidistant values of $\tau\in(-1,1)$ using (\ref{OpenAppr}) and the adaptive quadrature
{\quadgk} from the {\Matlab} system, and the graph of the numerical estimation of this
error (irregular light red line) computed according to (\ref{E2Def})--(\ref{EoDef}),
(\ref{AvgOpenEst}) with (\ref{D01fSimple}), and (\ref{logErr})--(\ref{der2Err}).}
\label{sinexp-gk}
\vspace*{-2ex}
\end{figure}

\begin{figure}[!ht]
\centering\vspace*{-0.0ex}
\includegraphics[width=14cm,height=5.5cm,angle=0]{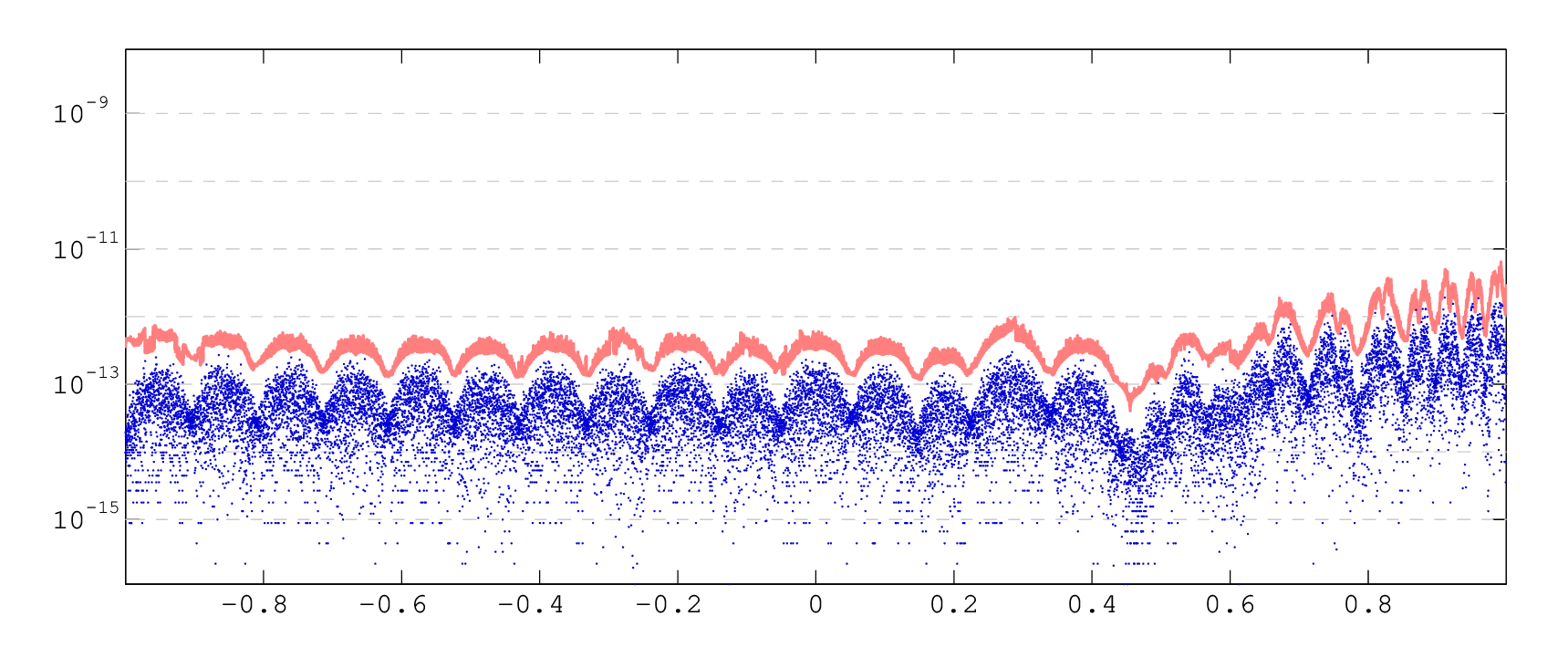}
\vspace{-3.5ex}
\caption{\small The graph of the absolute error (dark blue dots) of the
approximation to the integral (\ref{CPInt1}) for $f(x) = f_8(j(x))$, where
$f_8(x) = \sin(33x) + \exp(\sin(\e^{4x}))$ and $j(x)=\arcsin(\sin(2\pi+x))$, computed for
$19999$ equidistant values of $\tau\in(-1,1)$ using (\ref{OpenAppr}) and the adaptive quadrature
{\quadgk} from the {\Matlab} system, and the graph of the numerical estimation of this
error (irregular light red line) computed according to (\ref{E2Def})--(\ref{EoDef}),
(\ref{AvgOpenEst}) with (\ref{D01fSimple}), and (\ref{logErr})--(\ref{der2Err}).}
\label{sinexp_4e-gk}
\vspace*{2ex}
\end{figure}

In the final test involving the {\quadgk} adaptive quadrature,
we consider once again the function $f_5$ (\ref{f5-7Def}) which is
difficult for two reasons. It is a nearly singular function, and, what
is numerically very dangerous in this example, the constant $1.00001$ does
not have an accurate representation in the double precision arithmetic.
The graphs of the absolute error of the approximation to the integral
$I_{\tau}(f_5)$ and its estimation computed using the proposed
method are presented in Figure \ref{invx2-gk}.

In the case the computation of the values of the function $f$ in (\ref{CPInt1})
is very unstable, the proposed method may fail to compute a reliable error
estimation (as, practically, almost every other method). However, by the
presented examples, we may conclude that the proposed error estimates
seem to be a very reasonable practical choice.

\begin{figure}[!ht]
\centering\vspace*{-0.5ex}
\includegraphics[width=14cm,height=5.5cm,angle=0]{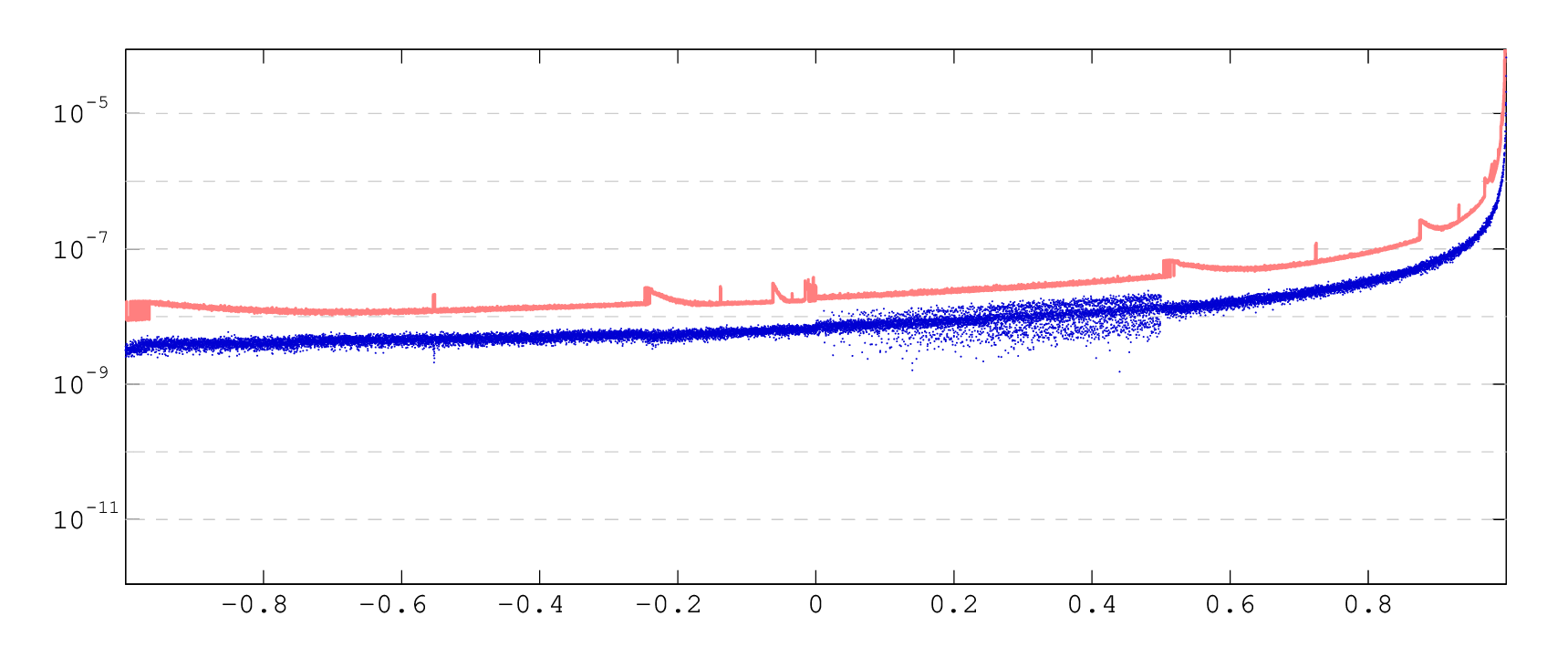}
\vspace{-3.5ex}
\caption{\small The graph of the absolute error (dark blue dots) of the approximation
to the integral (\ref{CPInt1}) for $f(x) = \frac{1}{100}(x-1.00001)^{-2}$, computed
for $19999$ equidistant values of $\tau\in(-1,1)$ using (\ref{OpenAppr}) and the adaptive
quadrature {\quadgk} from the {\Matlab} system, and the graph of the numerical estimation
of this error (irregular light red line) computed according to (\ref{E2Def})--(\ref{EoDef}),
(\ref{AvgOpenEst}) with (\ref{D01fSimple}), and (\ref{logErr})--(\ref{der2Err}).}
\label{invx2-gk}
\vspace*{1.75ex}
\end{figure}

Another adaptive quadrature we have tested our method with, is the {\quadcc}
algorithm based on the results given in \cite{GonnetAlg}. The {\quadcc} function is
not included in {\Matlab}, but can be downloaded from the {\Matlab} \emph{File Exchange}
webpage\footnote{\url{http://www.mathworks.com/matlabcentral/fileexchange/35489-quadcc}}.
The most significant difference between the {\quadgk} and the {\quadcc} adaptive
quadratures is that the latter uses a nonlinear error estimator. In the {\quadcc}
algorithm, the integral (\ref{Int}) (over a single subinterval) is approximated using
the $2^{k}+1$ point Clenshaw-Curtis quadratures for $k=2,3,4,5$. As the estimation
of the quadrature error, the L$^2$-norm of the difference between the integrand
and the corresponding interpolant is used (the norm is approximated by means
of Legendre polynomial expansion). Thus, considering the influence of
round-off errors, a possible perturbation of the integrand at a quadrature
node will more likely increase the value of the {\quadcc} error estimate rather
than partially cancel out with perturbations at other nodes (which
is quite often in case of linear error estimators).
This makes the adaptive quadrature very reliable, but
may cause a considerable efficiency drop if the
error tolerance is set too small.

The {\quadcc} adaptive quadrature includes relative error control, which is
a little inconvenient for us for the reasons described when presenting experiments
in {\Maple}. To solve the problem we may (as in the {\Maple} implementation of our
method) compute rough approximations to the two integrals in (\ref{MainFormula}),
and then use these values to ''convert'' the absolute error tolerance into the
relative one before calling the {\quadcc} subroutine. On the other hand, in
the {\quadcc} algorithm, the verification whether the relative tolerance
is met is done by checking the condition $E \leq T\hts|J|$,
where $E$ is the current approximation of the absolute error, $T$ is the
error tolerance, and $J$ is the most recent approximation of the integral.
Just by modifying this condition to $E \leq T$, we may very easily change the
relative error control to the absolute one. With both above solutions, we obtain
the same numerical results. However, with the latter (which we use in our
experiments) we save some computation time.

Formally, {\quadcc} is the closed type adaptive quadrature,
but it also features the properties of the open type one. This duality
is achieved by removing the node at which an infinite (or undefined)
value of the integrand is encountered from the set of nodes, and
then recalculating the quadrature weights. Therefore, using the
{\quadcc} subroutine, we also can apply the algorithm based on
(\ref{OpenAppr}) for approximating the integral (\ref{CPInt1}).
In the next two experiments, we are going to verify whether our
suspicion that the average case error estimate (\ref{AvgOpenEst})
may not be applicable for controlling the error tolerance of an adaptive
quadrature equipped with a nonlinear error estimator (cf.\ the
discussion at the end of Section \ref{SecErrOpenAvg}).

To this end, we have computed $19999$ values of the integral $I_{\tau}(f_9)$
(cf.\ (\ref{f9Def})) by the similar algorithm as when using the {\quadgk}
function. In the first experiment, the error tolerance for the {\quadcc} subroutine
was controlled by the pessimistic error estimate (\ref{PesOpenEst}) (under the
assumption that $x_{00}>\eps^{1/2}$), in the second one -- by the average case
estimate (\ref{AvgOpenEst})\footnote{As usual, we also include the estimates
(\ref{logErr})--(\ref{der2Err}).}. In both cases, the distribution of the actual
approximation error is very similar, while -- of course -- the computed error
estimation (which is often the only information for the user about the accuracy of the
approximation) is considerably different. The results are presented in Figure~\ref{sinexp_4e-3cc}.

As we could expect, the error tolerance based on the inequality (\ref{AvgOpenEst}) turned
out to be a little too small for the {\quadcc} algorithm. Switching from the pessimistic error
tolerance to the average case one caused a huge efficiency drop. The computation time
increased by\footnote{Note that the function $f_9$ is quite difficult for numerical
computation. In case of simpler integrands the above-mentioned difference
in the computation time is much smaller.} almost $1400\%$ (for comparison, 
when performing similar experiments involving the {\quadgk} subroutine which uses
the linear error estimator, the analogous growth of the computation time was only
$10\%$). However, we have observed that by only slightly releasing the error tolerance,
i.e.\ multiplying the right hand side of (\ref{AvgOpenEst}) by $2$, the increase
of the computation time may be considerably decreased (to only $250\%$ in
the considered example, which seems to be an acceptable price for
obtaining much more tight error estimation).

\begin{figure}[!ht]
\centering\vspace*{-0.0ex}
\includegraphics[width=14cm,height=5.5cm,angle=0]{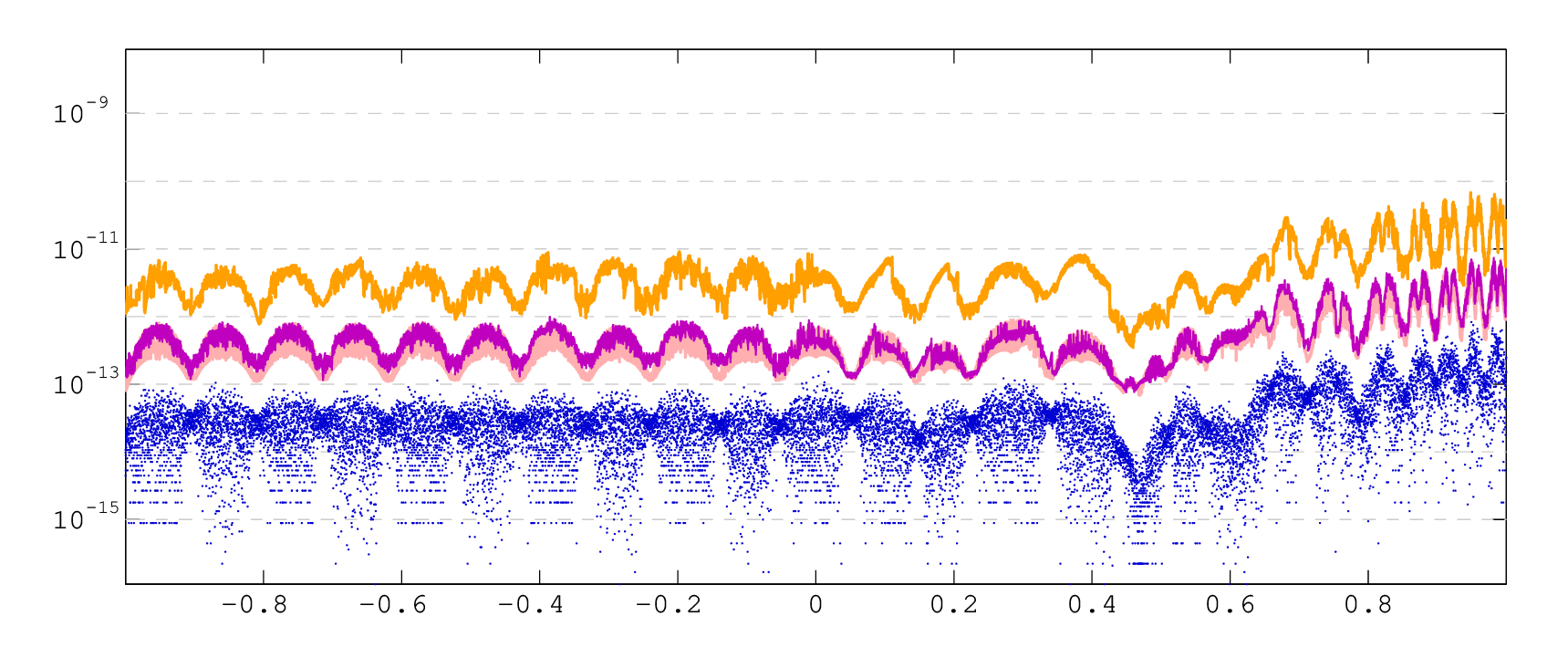}
\vspace{-3.5ex}
\caption{\small The graph of the absolute error (dark blue dots) of the
approximation to the integral (\ref{CPInt1}) with $f(x) = f_8(j(x))$, where
$f_8(x) = \sin(33x) + \exp(\sin(\e^{4x}))$ and $j(x)=\arcsin(\sin(2\pi+x))$, computed
for $19999$ equidistant values of $\tau\in(-1,1)$ using (\ref{OpenAppr}) and the adaptive
quadrature {\quadcc} \cite{GonnetAlg}, and the graph of the error estimation computed
using (\ref{E2Def})--(\ref{EoDef}), (\ref{AvgOpenEst}) with (\ref{D01fSimple}), and
(\ref{logErr})--(\ref{der2Err}) in the case the error tolerance was controlled by:
a) (\ref{PesOpenEst}) (irregular light orange line), b) (\ref{AvgOpenEst})
(irregular pale red line), c) (\ref{AvgOpenEst})$\times 2$ (irregular
purple line). The graph of the actual error corresponds to the case c),
however is very similar in both other cases.}
\label{sinexp_4e-3cc}
\vspace*{2ex}
\end{figure}

The proposed method based on the {\quadcc} adaptive quadrature is considerably
slower than the analogous one that uses the {\quadgk} subroutine (from a few to several
dozen times, depending on the regularity of the integrand). However, the error estimator
of the {\quadcc} subroutine, responsible for estimating the errors $E_A(\cdot,\cdot,\cdot)$
in (\ref{Appr1Err}) or (\ref{E2Def}), is extremely reliable. The proposed algorithm in
connection with the {\quadgk} quadrature offers a very reasonable level of reliability,
but examples can be found\footnote{E.g.,\ for more rapidly oscillating integrands.
Here, it is worth noting that if $f(x)$ in (\ref{CPInt1}) is of the form
$\varphi(x)\sin(cx)$ or $\varphi(x)\cos(cx)$, where $|c|\gg 1$ and $\varphi$
is an analytic (or smooth) function, then dedicated nonadaptive \cite{He},
\cite{Wang}, or automatic \cite{KellerCPOsc} methods are available.},
where the error estimator in the {\quadgk} algorithm fails, causing our method
to return incorrect (too small) error estimation. When testing the proposed
method with the {\quadcc} adaptive quadrature, we did not find an integral
of the form (\ref{CPInt1}) for which the approximation error would have
been falsely estimated\footnote{That does not mean such examples do not
exist. Knowing the nodes of the quadrature $Q$ in (\ref{QDef}), we may always
construct an example for which the adaptive quadrature reports an incorrect
error estimation (cf.\ \cite[\S5.2]{ShampineBook}).}.

While presenting the results of {\Matlab} experiments, we
have concentrated on the method presented in Section \ref{SecErrOpen}.
With any reliable, open or closed type, adaptive quadrature we may successfully
apply the algorithm based on (\ref{1stAppr})--(\ref{Appr1Err}). The results of such
an experiment for the adaptive quadrature {\quadcc} in the case of the integral
$I_{\tau}(f_9)$ are presented in Figure \ref{sinexp_4e-gku} (when applying
the {\quadgk} subroutine, the results are almost identical). The distance between
the approximation error and its estimation is larger than $0.35$ of decimal digit for
all values of $\tau$. As we can also see, the graphs of the error estimation and the actual
error are very regular. This is because the most significant part of the approximation
error is related to the removal of the interval $[\tau-\mu,\tau+\mu]$ from the
interval of integration (the last term in (\ref{Appr1Err})). For the same
reason, the approximation error is larger than the one we obtain when
applying the approach presented in Section \ref{SecErrOpen}.

\begin{figure}[!ht]
\centering\vspace*{-0.0ex}
\includegraphics[width=14cm,height=5.5cm,angle=0]{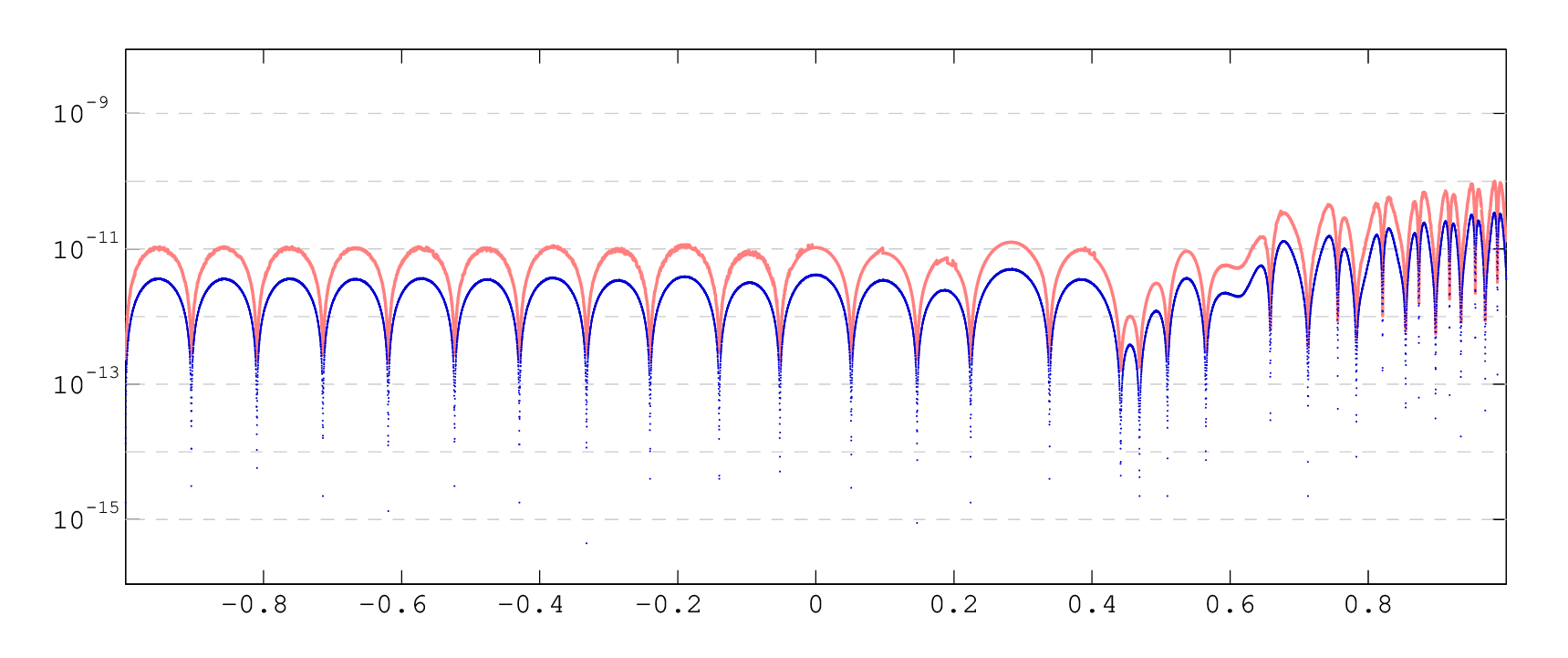}
\vspace{-3.5ex}
\caption{\small The graph of the absolute error (dark blue dots) of the
approximation to the integral (\ref{CPInt1}) with $f(x) = f_8(j(x))$, where
$f_8(x) = \sin(33x) + \exp(\sin(\e^{4x}))$ and $j(x)=\arcsin(\sin(2\pi+x))$,
computed for $19999$ equidistant values of $\tau\in(-1,1)$ using (\ref{1stAppr}) and the 
adaptive quadrature {\quadcc}, and the graph of the error estimation (red solid line)
computed according to (\ref{Appr1Err}) (with $\mu$ satisfying (\ref{muCond})),
(\ref{D01fSimple}), and (\ref{logErr})--(\ref{der2Err}).}
\label{sinexp_4e-gku}
\vspace*{2ex}
\end{figure}

%%%%%%%%%%%%%%%%%%%%%%%%%%%%%%%%%%%%%%%%%%%%%%%%%%%%%%%%%%%%%%%%%%%%%%%%%%%%%%%%
\subsection{The present method versus {\qawc}} \label{SecNumQawc}

In this paper, we have applied a standard (designed for evaluating the integrals
of the form (\ref{Int})) adaptive quadrature to compute the Cauchy principal
value integrals. An adaptive quadrature designed specially for computing
the integrals (\ref{CPInt}) was included in Quadpack \cite{Quadpack}.
The subroutine is called {\qawc} and is still used quite frequently.

The {\qawc} algorithm uses two types of quadratures. The $25$-point generalized
Clenshaw-Curtis quadrature for evaluating the integral in the subinterval
containing the singularity point, and the well known $15$-point Kronrod
extension of the $7$-point Gauss-Legendre quadrature. The algorithm includes
empirical round-off error detection. If the estimated approximation error is
too large compared to the value of the integral in too many subintervals,
the {\qawc} subroutine reports a failure. For the experiment, we
have used the \emph{Matlab} translation of {\qawc} from the
\emph{Slatec} library\footnote{The library can be downloaded from
\url{http://www.mathworks.com/matlabcentral/fileexchange/14535-slatec}.}.

We have tested the {\qawc} adaptive subroutine for three different
functions $f$ in (\ref{CPInt1}): $f\equiv f_5$ (\ref{f5-7Def}),
$f\equiv f_9$ (\ref{f9Def}), and $f\equiv f_{10}(x)=100\hts(x+1/2)^2$.
The error tolerance was set to $\sqrt{\eps}$ in the case of the function $f_5$,
and to $1000\eps$, otherwise. The assumed error tolerances are a little too small to
be reached, so we can verify how well the round-off errors are detected by the {\qawc}
algorithm. The results are not satisfactory (see Figure \ref{fx4-qawc}). For the
integrals $I_{\tau}(f_5)$ and $I_{\tau}(f_{10})$, the approximation error was
falsely estimated for $15171$ and $9056$ of $19999$ values of $\tau$, respectively.
The best results were obtained in the case the integral $I_{\tau}(f_9)$,
the approximation error was estimated incorrectly for only $144$ values of
$\tau$, and in $1694$ cases the {\qawc} function reported a failure. For higher
error tolerances, the results improved, but only in the case of functions $f_5$ and
$f_9$. We have found that, in general, the {\qawc} algorithm computes quite reliable
results, if a correct error tolerance $T_{\qawc}$ is set, and the absolute error is estimated
by $\max\{T_{\qawc},E_{\qawc}\}$ instead of $E_{\qawc}$, where $E_{\qawc}$ corresponds
to the error estimate reported by the {\qawc} function. In the last graph of Figure
\ref{fx4-qawc}, we present the results of the experiment involving the integral
$I_{\tau}(f_9)$ again. This time however, the proper error tolerance
$T_{\qawc}$ was set according to (\ref{AvgOpenEst})%
\footnote{Note that the {\qawc} algorithm reports a failure by testing
the influence of round-off errors locally, in separate (sometimes very small)
subintervals. In this paper, we gave the global estimation of these errors.
In order to prevent the {\qawc} subroutine from reporting a failure which
in fact did not occur, we had to use a little less tight error estimate,
i.e.\ (\ref{AvgOpenEst}) and (\ref{DfSimple}) with $D_{1,f}$ and $D_{0,f}$
defined as in (\ref{D1Def}) and (\ref{D0Def}).}
and (\ref{tauErr})--(\ref{der2Err}), while the approximation
error was estimated as described above. The integral was computed
correctly for all values of $\tau$. With this scheme, the correct
results were also obtained in the case of the other tested integrals.
As far as the efficiency is concerned, the {\qawc} subroutine turned out
to be slower than the proposed method using the {\quadgk} or
the {\quadcc} adaptive quadrature.

\begin{figure}[!ht]
\centering\vspace*{-0.0ex}
\includegraphics[width=14cm,height=5.5cm,angle=0]{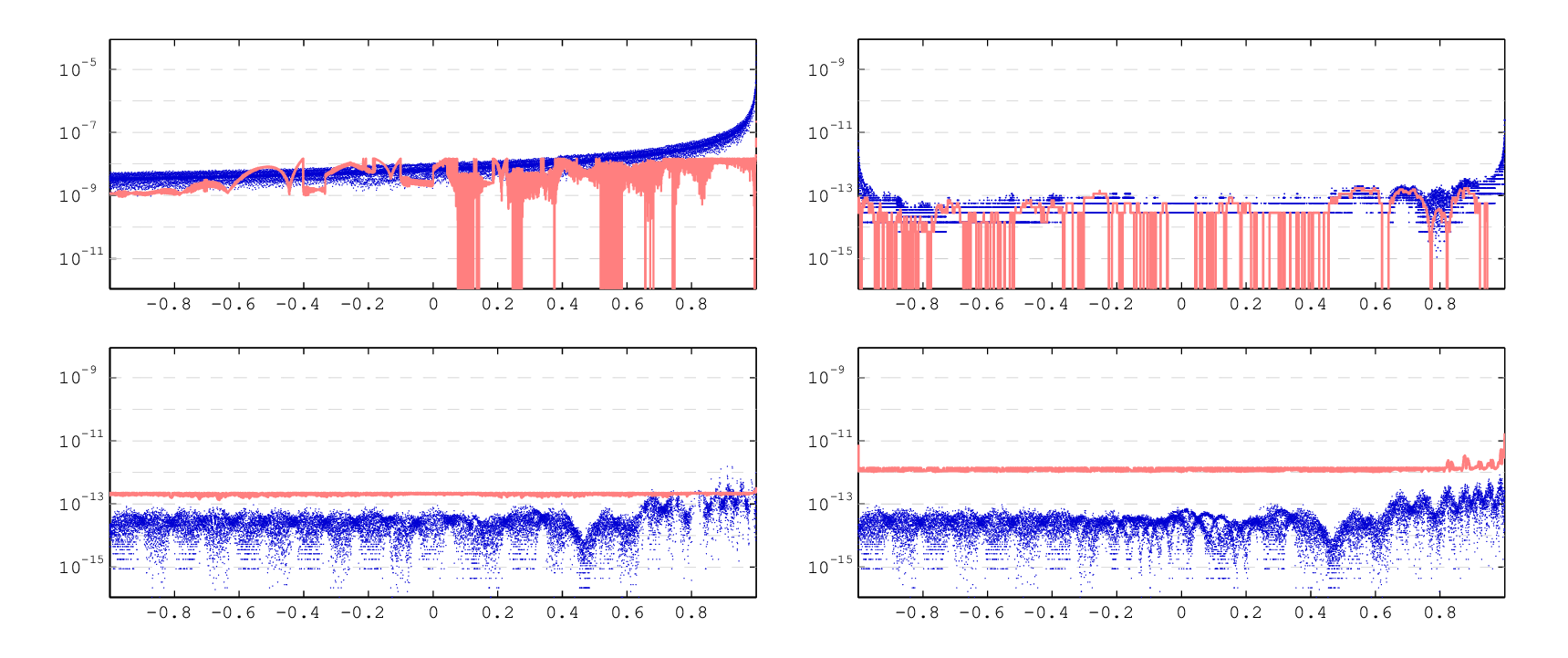}
\vspace{-3.5ex}
\caption{\small The graphs of the absolute error (dark blue dots) of
the approximation to the integral (\ref{CPInt1}), and the graphs
of the error estimation (irregular light red line) computed
by the {\qawc} algorithm from Quadpack \cite{Quadpack} in the case
a)~$f(x)=f_5(x)$ (\ref{f5-7Def}) and the error tolerance $T_{\qawc}=\sqrt{\eps}$ (upper-left graph),
b) $f(x)=100\hts(x+1/2)^2$ and the error tolerance $T_{\qawc}=1000{\eps}$ (upper-right graph),
c) $f(x)=f_9(x)$ (\ref{f9Def}) and the error tolerance $T_{\qawc}=1000{\eps}$ (lower-left graph).
The last graph also presents the results for $f\equiv f_9$, but in the case the error tolerance
and the error estimation of the {\qawc} subroutine were supported by the estimates
proposed in this paper.}
\label{fx4-qawc}
\vspace*{0ex}
\end{figure}

%%%%%%%%%%%%%%%%%%%%%%%%%%%%%%%%%%%%%%%%%%%%%%%%%%%%%%%%%%%%%%%%%%%%%%%%%%%%%%%%
%%%%%%%%%%%%%%%%%%%%%%%%%%%%%%%%%%%%%%%%%%%%%%%%%%%%%%%%%%%%%%%%%%%%%%%%%%%%%%%%
\section{Conclusions and final thoughts} \label{SecFinal}
\setcounter{equation}{0}
%%%%%%%%%%%%%%%%%%%%%%%%%%%%%%%%%%%%%%%%%%%%%%%%%%%%%%%%%%%%%%%%%%%%%%%%%%%%%%%%
%%%%%%%%%%%%%%%%%%%%%%%%%%%%%%%%%%%%%%%%%%%%%%%%%%%%%%%%%%%%%%%%%%%%%%%%%%%%%%%%

We have shown that a fine ordinary adaptive quadrature can be successfully used
for fast and accurate computations of the Cauchy principal value integrals, if the
error tolerance for the adaptive quadrature and the reported error estimation
remain under control of the error estimates proposed in this paper.

The presented method works most efficiently when used with an adaptive quadrature
using a linear error estimator. Such adaptive quadratures, however, may sometimes accidentally
fail to compute a reliable approximation. The probability of such a failure may be considerably
decreased by setting very small error tolerance. On the other hand, if the tolerance is too
small, the result may be quite the opposite, as also the efficiency of the adaptive quadrature
may drop significantly. In the present paper, we give tools that enable the automatic
selection of ''near optimal'' error tolerance, ensuring high level
of reliability and high efficiency of the proposed algorithm.

A small disadvantage of the presented method is that some knowledge on the given
adaptive quadrature is required for a proper selection of constants the proposed
error estimates are based on. In addition, not every adaptive quadrature can be used
together with the proposed method. Such a quadrature should exhibit a stable behaviour
in the presence of round-off errors. We do not require the adaptive quadrature to
correctly estimate the influence of these errors (the error estimates derived in
this paper are meant to do so). The quadrature, however, should not miss the
accurate result by considerably more than the influence
of the round-off errors justifies.

% ============================= ACKNOWLEDGEMENTS ============================= %

%\section*{Acknowledgements}

%...

%%%%%%%%%%%%%%%%%%%%%%%%%%%%%%%%%%%%%%%%%%%%%%%%%%%%%%%%%%%%%%%%%%%%%%%%%%%%%%%%
%%%%%%%%%%%%%%%%%%%%%%%%%%%%%%%%%%%%%%%%%%%%%%%%%%%%%%%%%%%%%%%%%%%%%%%%%%%%%%%%

%\appendix

%%%%%%%%%%%%%%%%%%%%%%%%%%%%%%%%%%%%%%%%%%%%%%%%%%%%%%%%%%%%%%%%%%%%%%%%%%%%%%%%
%\section{Upgrading the algorithm of Hasegawa and Torii}

%...

%%%%%%%%%%%%%%%%%%%%%%%%%%%%%%%%%%%%%%%%%%%%%%%%%%%%%%%%%%%%%%%%%%%%%%%%%%%%%%%%
%%%%%%%%%%%%%%%%%%%%%%%%%%%%%%%%%%%%%%%%%%%%%%%%%%%%%%%%%%%%%%%%%%%%%%%%%%%%%%%%

% ================================= REFERENCES =============================== %

% ==================================================================== %

\end{document}